\documentclass[reqno]{amsart}

\usepackage{amsmath,amssymb,amsthm,amsfonts,enumerate, enumitem,hyperref,cleveref}
\usepackage{dsfont}
\usepackage[all]{xy}
\usepackage[T1]{fontenc}
\usepackage{tikz}
\usepackage{pgfplots}
\pgfplotsset{compat=1.15}
\usepackage{mathrsfs}
\usetikzlibrary{arrows}

\theoremstyle{plain}
\newtheorem*{theorem}{Theorem}
\newtheorem{thrm}{Theorem}[section]
\newtheorem{cor}[thrm]{Corollary}
\newtheorem{prop}[thrm]{Proposition}
\newtheorem{lem}[thrm]{Lemma}

\theoremstyle{definition}
\newtheorem{defn}[thrm]{Definition}
\newtheorem{rem}[thrm]{Remark}
\newtheorem{exm}[thrm]{Example}

\crefrangeformat{equation}{#3(#1)#4--#5(#2)#6}

\crefname{thrm}{Theorem}{Theorems}
\crefname{theorem}{Theorem}{Theorems}
\crefname{lem}{Lemma}{Lemmas}
\crefname{cor}{Corollary}{Corollaries}
\crefname{prop}{Proposition}{Propositions}
\crefname{defn}{Definition}{Definitions}
\crefname{exm}{Example}{Examples}
\crefname{rem}{Remark}{Remarks}
\crefname{conj}{Conjecture}{Conjectures}
\crefname{quest}{Question}{Questions}
\crefname{section}{Section}{Sections}
\crefname{equation}{\unskip}{\unskip}
\crefname{enumi}{\unskip}{\unskip}
\crefname{subsection}{Subsection}{Subsections}

\newcommand{\0}{\theta}
\newcommand{\ve}{\varepsilon}

\newcommand{\af}{\alpha}
\newcommand{\bt}{\beta}
\newcommand{\lb}{\lambda}

\newcommand{\vf}{\varphi}

\newcommand{\dl}{\delta}

\newcommand{\sg}{\sigma}

\newcommand{\B}{\mathcal{B}}

\newcommand{\Z}{\mathbb{Z}}

\newcommand{\m}{{}^{-1}}
\newcommand{\sst}{\subseteq}
\newcommand{\impl}{\Rightarrow}
\newcommand{\gen}[1]{\left\langle #1\right\rangle}
\newcommand{\lf}{\lfloor}
\newcommand{\rf}{\rfloor}

\newcommand{\ch}{\mathrm{char}}

\newcommand{\id}{\mathrm{id}}
\renewcommand{\iff}{\Leftrightarrow}

\begin{document}
	\title[Potent preservers of incidence algebras]{Potent preservers of incidence algebras}	
	
	\author{Jorge J. Garc{\' e}s}
	\address{ Departamento de Matem{\' a}tica Aplicada a la Ingenier{\' i}a Industrial, ETSIDI, Universidad Polit{\' e}cnica de Madrid, Madrid, Spain}
	\email{j.garces@upm.es}
	
	\author{Mykola Khrypchenko}
	\address{Departamento de Matem\'atica, Universidade Federal de Santa Catarina,  Campus Reitor Jo\~ao David Ferreira Lima, Florian\'opolis, SC, CEP: 88040--900, Brazil}
	\email{nskhripchenko@gmail.com}

	
	\begin{abstract}
		Let $X$ be a finite connected poset, $F$ a field and $I(X,F)$ the incidence algebra of $X$ over $F$. We describe the bijective linear idempotent preservers $\vf:I(X,F)\to I(X,F)$. Namely, we prove that, whenever $\ch(F)\ne 2$, $\vf$ is either an automorphism or an anti-automorphism of $I(X,F)$. If $\ch(F)=2$ and $|F|>2$, then $\vf$ is a (in general, non-proper) Lie automorphism of $I(X,F)$. Finally, if $F=\Z_2$, then $\vf$ is the composition of a bijective shift map and a Lie automorphism of $I(X,F)$. Under certain restrictions on the characteristic of $F$ we also obtain descriptions of the bijective linear maps which preserve tripotents and, more generally, $k$-potents of $I(X,F)$ for $k\ge 3$.
	\end{abstract}
	
	\maketitle
	
	\tableofcontents
	
	\section*{Introduction}
	
	By a {\it preserver} on a (not necessarily associative) algebra $A$ one usually means a map on $A$ (possibly, with values in another algebra $B$) that preserves some fixed property of an element of $A$ (e.g. idempotency~\cite{Beasley-Pullman91,Dolinar03}, nilpotency~\cite{Botta-Pierce-Watkins83}, invertibility~\cite{Marcus-Purves59,Pazzis10}, non-invertibility~\cite{Dieudonne49}) or some fixed relation between elements of $A$ (e.g. orthogonality~\cite{BFPGMP08}, commutativity~\cite{Watkins76,Omladic-Radjavi-Semrl01,Bresar-Semrl06}, anti-commutativity~\cite{Fosner-KK-Sze11}). If $A$ is a subalgebra of $M_n(F)$, one can also consider the maps on $A$ preserving the determinant~\cite{Frobenius97,Dolinar-Semrl02}, permanent~\cite{Budrevich-Guterman-Duffner19}, rank~\cite{Beasley70,Beasley81} etc.
	
	
	Potent preservers on matrix algebras have been studied since 1980's. Chan, Lim and Tan~\cite{Chan-Lim-Tan87} proved that any idempotent preserver on $M_n(\mathbb C)$ or $M_n(\mathbb R)$ fixing the identity matrix is either an automorphism or an anti-automorphism. Beasley and Pullman~\cite{Beasley-Pullman91} described the semigroup of linear idempotent preservers of $M_n(F)$ over an arbitrary field $F$. They proved that, whenever $\ch(F)\ne 2$, this semigroup is generated by the transposition and the similarity maps, i.e. it coincides with the group of automorphisms and anti-automorphisms of $M_n(F)$. If $\ch(F)=2$, then the shift operators also appear among the generators. In~\cite{Chan-Lim92} Chan and Lim characterized tripotent preservers of $M_n(F)$, where $\ch(F)\not\in\{2,3\}$. For bijective linear maps on $M_n(F)$ over an algebraically closed field $F$ a much more general fact was proved by Howard in~\cite{Howard80}. He described those of them which preserve $A\in M_n(F)$ satisfying $P(A)=0$, where $P$ is a polynomial over $F$ admitting at least $2$ distinct roots. Bre\v{s}ar and \v{S}emrl~\cite{Bresar-Semrl93b} studied non-zero potent preservers on $M_n(\mathbb{C})$ without fixing the potency index. It turns out that such maps are $k$-potent preservers for some $k$ and thus they are automorphisms or anti-automorphisms of $M_n(\mathbb{C})$ multiplied by a $(k-1)$-th root of unity $c\in\mathbb{C}$.
	
	It follows from \cite[Corollary 4 and Theorem 5]{Molnar-Semrl98} that automorphisms and anti-automorphisms are the only bijective linear idempotent preservers of the upper-triangular matrix algebra $T_n(\mathbb C)$. Xu, Cao and Tang~\cite{Xu-Cao-Tang07} generalized this result to the case of an arbitrary field. As it was for $M_n(F)$, in characteristic $2$ shift maps appear in the description. On the other hand, Wang and You~\cite{Wang-You08} described $k$-potent preservers on $T_n(\mathbb{C})$ with values in $M_n(\mathbb{C})$. There also exists~\cite{Slowik14} a (quite complicated) characterization of idempotent preservers on the algebra $T_\infty(F)$ of infinite upper-triangular matrix algebras, where $\ch(F)\ne 2$.
	
	The incidence algebra $I(X,F)$ of a finite connected poset $X$ over a field $F$ is a natural generalization of $T_n(F)$. In this paper we study the bijective linear maps $I(X,F)\to I(X,F)$ preserving $k$-potents. The main part of our work is devoted to idempotent preservers ($k=2$), where our results can be gathered into the following theorem.
	
	\begin{theorem}
		Let $X$ be a finite connected poset, $F$ a field and $\vf:I(X,F)\to I(X,F)$ a bijective $F$-linear map.
		\begin{enumerate}
			\item If $\ch(F)\ne 2$, then $\vf$ is an idempotent preserver $\iff\vf$ is either an automorphism of $I(X,F)$ or an anti-automorphism of $I(X,F)$.
			\item If $\ch(F)=2$ and $|F|>2$, then $\vf$ is an idempotent preserver $\iff\vf$ is a Lie automorphisms of $I(X,F)$ of the form $\psi\circ\tau_{\0,\sg,c}$, where $\psi$ is an inner automorphism of $I(X,F)$ and $\tau_{\0,\sg,c}$ is an elementary Lie automorphism of $I(X,F)$ in which $c_i\in\{0,1\}$ for all $i$. 
			\item If $F=\Z_2$, then $\vf$ is an idempotent preserver $\iff\vf$ is a composition of a bijective shift map and a Lie automorphism of $I(X,F)$.
		\end{enumerate}
	\end{theorem}
	Based on the description of idempotent preservers, we characterize $k$-potent preservers of $I(X,F)$ for $k\ge 3$, but under certain hypotheses on $F$.
	
	More precisely, our paper is organized as follows. In \cref{sec-prelim} we give definitions, fix notations, recall some well-known facts and prove two elementary statements. In \cref{sec-2-torsion-free} we deal with the easy case of algebras over $2$-torsion-free rings. We adapt the proofs of~\cite[Theorem 5]{Molnar-Semrl98} and~\cite[Theorem 2.1]{Bresar-Semrl93a} to incidence algebras and show in \cref{Idemp-pres-is-Jordan-homo} that any idempotent preserver on $I(X,R)$ is a Jordan homomorphism. We start  \cref{sec-char(F)=2-and-|F|>2} by observing in \cref{Idemp-pres-over-field-with>2-elem} that \cref{Idemp-pres-is-Jordan-homo} also holds over fields of characteristic $2$ with more than $2$ elements. In the latter case the description of bijective linear idempotent preservers of $I(X,F)$ reduces to the description of idempotent-preserving elementary Lie automorphisms~\cite{FKS} of $I(X,F)$, which is accomplished in \cref{tau_0_sg_c-idemp-pres<=>c_i=0-or-1}. In \cref{sec-Jordan-auto} we characterize Jordan automorphisms of the incidence algebra $I(X,F)$ of a finite connected poset $X$ over a field $F$ (see \cref{Jord-auto-is-auto-or-anti-auto}), which together with the result of \cref{sec-2-torsion-free} gives the description of bijective linear idempotent preservers of $I(X,F)$, whenever $\ch(F)\ne 2$ (see \cref{bijective-idemp-pres-is-auto-or-anti-auto}). We then proceed in \cref{sec-I(X_Z_2)} with the remaining and the most difficult case $F=\Z_2$. It turns out that bijective linear idempotent preservers of $I(X,\Z_2)$ are exactly Lie automorphisms of $I(X,\Z_2)$ composed with bijective shift maps, as \cref{idemp-pres-of-I(X_Z_2)} states. We apply in \cref{sec-trip-pres} our description of idempotent preservers of $I(X,F)$ to characterize tripotent preservers of $I(X,F)$ provided that $\ch(F)\not\in\{2,3\}$ (see \cref{trip-pres-description}). The main trick here is the new associative product $\bullet_u$ on the image of a tripotent preserver $\vf$ which makes $\vf$ an idempotent preserver with respect to $\bullet_u$. The same idea is used in \cref{sec-k-potent-pres} for $k$-potent preservers of $I(X,F)$ with $k\ge 3$, however we chose to consider $k=3$ in \cref{sec-trip-pres} separately, since its intermediate results (\cref{p trip pres is triple homo,l T(1) unitary}) are themselves interesting and do not generalize to $k>3$. The final fact (\cref{k-potent-pres-description}), though, is a generalization of \cref{trip-pres-description} proved in \cref{sec-trip-pres} for tripotent preservers. In Appendix we give a description of $k$-potents of $I(X,F)$, which is an interesting related result that was not used in the rest of the paper.  
	
	\section{Preliminaries}\label{sec-prelim}
	
	\subsection{Rings and algebras}
	An element $a$ of a ring $R$ is called an {\it idempotent} (resp. \textit{tripotent}), if $a^2=a$ (resp. $a^3=a
	$). We denote by $E(R)$ (resp. $T(R)$) the set of idempotents (resp. tripotents) of $R$. More generally, given an integer $k\ge 2$, if $a^k=a$, then $a$ is called a {\it $k$-potent}. The set of $k$-potents of $R$ will be denoted by $P_k(R)$. For any $A\subseteq R$ we set $E(A):=E(R)\cap A$. 
	
	The {\it centralizer} of $A\sst R$ in $R$ is $C_R(A)=\{r\in R\mid ar=ra\text{ for all }a\in A\}$. If $R$ is associative, then $C_R(A)$ is a subring of $R$. In particular, the {\it center of $R$} is $C(R):=C_R(R)$. 
	
	Let $n$ be a positive integer. A ring $R$ is said to be {\it $n$-torsion-free} if $nr=0$ holds only for $r=0$. A ring $R$ {\it has characteristic $n$} ($\ch(R)=n$) if $nr=0$ for all $r\in R$, and $n$ is the smallest number with this property. If such a number $n$ does not exist, then $\ch(R):=0$. Clearly, $\ch(R)=n$ implies $R$ is not $n$-torsion-free, and moreover, any $R$-algebra is not $n$-torsion-free. If $F$ is a field, then $F$ is $n$-torsion-free $\iff$ any $F$-algebra is $n$-torsion-free $\iff$ $\ch(F)\nmid n$.
	
	We say that $a,b\in R$ are {\it orthogonal} if $ab=ba=0$. An idempotent $e\in E(R)$ is {\it primitive} if it cannot be written as a sum of two non-zero orthogonal idempotents.
	
	\subsection{Elementary linear algebra}
	
	The following two facts might be well-known, but we could not find an appropriate reference.
	
	\begin{lem}\label{distr-direct-sum}
		Let $M$ be an $R$-module, $N_i\sst M$, $i\in I$, a family of $R$-submodules of $M$ and $L\sst M$ an $R$-submodule of $M$ such that $L\cap\left(\sum_{i\in I}N_i\right)=\{0\}$. Then $\bigcap_{i\in I}(L\oplus N_i)=L\oplus\left(\bigcap_{i\in I}N_i\right)$.
	\end{lem}
	\begin{proof}
		The direct sums make sense because $L\cap\left(\sum_{i\in I}N_i\right)=\{0\}$. Moreover, the inclusion $L\oplus\left(\bigcap_{i\in I}N_i\right)\sst\bigcap_{i\in I}(L\oplus N_i)$ is obvious. For the converse inclusion take $m\in \bigcap_{i\in I}(L\oplus N_i)$. Then $m=l_i+n_i$, where $l_i\in L$ and $n_i\in N_i$, $i\in I$. For all $i,j\in I$ we have $l_i-l_j=n_j-n_i\in L\cap\left(\sum_{i\in I}N_i\right)$, so $l_i=l_j$ and $n_i=n_j$. It follows that $m=l+n$, where $l\in L$ and $n\in\bigcap_{i\in I}N_i$. Thus, $m\in L\oplus\left(\bigcap_{i\in I}N_i\right)$.
	\end{proof}
	
	For an $R$-module $M$ and a subset $X\sst M$ we denote by $\gen{X}$ the $R$-submodule of $M$ generated by $X$, where $\gen{\emptyset}=\{0\}$. A subset $X\sst M$ is said to be {\it linearly independent} (over $R$) if for any $\{x_i\}_{i=1}^n\sst X$ and $\{r_i\}_{i=1}^n\sst R$ it follows from $\sum_{i=1}^nr_ix_i=0$ that $r_i=0$ for all $1\le i\le n$.
	
	\begin{lem}\label{inter-spans}
		Let $M$ be an $R$-module and $X\sst M$ linearly independent. Then for all $A,B\sst X$ we have $\gen{A}\cap\gen{B}=\gen{A\cap B}$.
	\end{lem}
	\begin{proof}
		The inclusion $\gen{A\cap B}\sst\gen{A}\cap\gen{B}$ is obvious. Let now $m\in \gen{A}\cap\gen{B}$. Then $m=\sum_{i=1}^nr_ia_i=\sum_{j=1}^ks_jb_j$ for some $\{a_i\}_{i=1}^n\sst A$, $\{b_j\}_{j=1}^k\sst B$ and $\{r_i\}_{i=1}^n,\{s_j\}_{j=1}^k\sst R$. If $\{a_i\}_{i=1}^m\cap\{b_j\}_{j=1}^k=\emptyset$, then $\sum_{i=1}^nr_ia_i-\sum_{j=1}^ks_jb_j=0$ implies $r_i=0$ for all $1\le i\le n$ and $s_j=0$ for all $1\le j\le k$ by the linear independence of $X$. Hence, $m=0\in\gen{A\cap B}$. Let $\{a_i\}_{i=1}^n\cap\{b_j\}_{j=1}^k\ne\emptyset$ and assume that $\{a_i\}_{i=1}^n$ and $\{b_j\}_{j=1}^k$ are enumerated in a way that $a_i=b_i$ for all $1\le i\le l$ and $\{a_i\}_{i=l+1}^n\cap\{b_j\}_{j=l+1}^k=\emptyset$ (in particular, $\{a_i\}_{i=l+1}^n$ or $\{b_j\}_{j=l+1}^k$ or both of them may be empty). Then $\sum_{i=1}^l(r_i-s_i)a_i+\sum_{i=l+1}^nr_ia_i-\sum_{j=l+1}^ks_jb_j=0$, so $r_i=s_i$ for all $1\le i\le l$, $r_i=0$ for all $l+1\le i\le n$ and $s_j=0$ for all $l+1\le j\le k$ in view of the linear independence of $X$. Thus, $m=\sum_{i=1}^lr_ia_i\in\gen{A\cap B}$.
	\end{proof}
	\subsection{Potent preservers}
	
	Let $A$ and $B$ be algebras over a commutative unital ring $R$. We say that an $R$-linear map $\vf:A\to B$ \textit{preserves idempotents (resp. tripotents)} (or is an {\it idempotent (resp. tripotent) preserver}) if $\vf(E(A))\sst E(B)$ (resp. $\vf(T(A))\sst T(B)$). More generally, $\vf$ is a {\it $k$-potent preserver} for some integer $k\ge 2$, if $\vf(P_k(A))\sst P_k(B)$. An idempotent preserver is {\it strong} if $\vf(A\setminus E(A))\sst B\setminus E(B)$. Clearly, any homomorphism or anti-homomorphism $A\to B$ preserves $k$-potents for all $k\ge 2$. Moreover, isomorphisms and anti-isomorphisms are strong idempotent preservers.
	
	\subsection{Jordan maps}
	A \textit{Jordan homomorphism}~\cite{Jacobson-Rickart50} between associative algebras $A\to B$ is an $R$-linear map $\vf:A\to B$ satisfying
	\begin{align}
	\vf(a^2)&=\vf(a)^2,\label{vf(a^2)}\\
	\vf(aba)&=\vf(a)\vf(b)\vf(a)\label{vf(aba)}
	\end{align}
	for all $a,b\in A$. It follows from \cref{vf(aba)} that
	\begin{align}\label{vf(abc+cba)}
	\vf(abc+cba)=\vf(a)\vf(b)\vf(c)+\vf(c)\vf(b)\vf(a)
	\end{align}
	for all $a,b,c\in A$.
	Observe that \cref{vf(a^2)} implies	
	\begin{align}\label{vf(ab+ba)}
	\vf(ab+ba)=\vf(a)\vf(b)+\vf(b)\vf(a)
	\end{align}
	for all $a,b\in A$, i.e. $\vf$ preserves the {\it Jordan product} $a\circ b=ab+ba$. If $B$ is $2$-torsion-free, then \cref{vf(ab+ba)} implies \cref{vf(a^2),vf(aba)}, so Jordan homomorphisms are precisely the linear maps preserving the Jordan product. However, for instance if $\ch(R)=2$, there could exist maps preserving the Jordan product, but not preserving the squares, thus we need to distinguish between the two notions in this setting. We shall use the term {\it Lie homomorphism} for a map preserving the Jordan product in the case $\ch(R)=2$, because then $a\circ b$ coincides with the {\it Lie product} $[a,b]=ab-ba$.
	
	Any Jordan homomorphism is an idempotent preserver by \cref{vf(a^2)}. If only \cref{vf(aba)} holds for all $a,b\in A$, then $\varphi$ is called a \emph{Jordan triple homomorphism}. Any Jordan triple homomorphism is a tripotent preserver.
	
	
	
	
	\subsection{Posets}
	
	A \textit{poset} is a set with a partial order (i.e., a reflexive, transitive and anti-symmetric binary relation) $\le$ on it. Given $x,y\in X$, the subset $\lf x,y\rf :=\{z\in X \mid x\le z\le y\}$ is called the \emph{interval} from $x$ to $y$. A poset $X$ is said to be \textit{locally finite} if all $\lf x,y\rf\sst X$ are finite. A \textit{chain} in $X$ is a non-empty subset $C\sst X$ such that $x\le y$ or $y\le x$ for all $x,y\in C$. The \textit{length} of a finite chain $C\sst X$ is defined to be $|C|-1$. The \textit{length} of a non-empty poset $X$, denoted by $l(X)$, is the maximum length of chains $C\sst X$ (if it exists). Obviously, any finite poset $X$ has length $l(X)\le |X|-1$. A {\it walk} in $X$ is a sequence $x_0,x_1,\dots,x_m$ of elements of $X$, such that for all $i=0,\dots,m-1$ either $x_i\le x_{i+1}$ and $l(\lf x_i,x_{i+1}\rf)=1$ or $x_{i+1}\le x_i$ and $l(\lf x_{i+1},x_i\rf)=1$. A poset $X$ is {\it connected} if for all $x,y\in X$ there exists a walk $x=x_0,\dots,x_m=y$. 
	
	A map $\lb:X\to Y$ between two posets is {\it order-preserving} (resp. {\it order-reversing}) whenever $x\le y\impl\lb(x)\le\lb(y)$ (resp. $x\le y\impl\lb(y)\le\lb(x)$) for all $x,y\in X$. An {\it (order) automorphism} (resp. {\it anti-automorphism}) of a poset $X$ is an order-preserving (resp. order-reversing) bijection of $X$ whose inverse is also order-preserving (resp. order-reversing).
	
	\subsection{Incidence algebras}
	
	The \emph{incidence algebra} $I(X,R)$ of a locally finite poset $X$ over a commutative unital ring $R$ is the $R$-module 
	$I(X,R)=\{f:X\times X\to R \mid f(x,y)=0 \text{ if } x\nleq y\}$ with the convolution product
	$$
	(fg)(x,y)=\sum_{x\leq z\leq y}f(x,z)g(z,y),
	$$
	where $f,g\in I(X,K)$. Under this operation $I(X,R)$ is an associative $R$-algebra with identity $\delta$ given by
	\begin{align*}
	\delta(x,y)=
	\begin{cases}
	1, & x=y,\\
	0, & x\ne y.
	\end{cases}
	\end{align*}
	The \textit{diagonal} of $f\in I(X,R)$ is $f_D\in I(X,R)$ such that
	\begin{align*}
	f_D(x,y)=
	\begin{cases}
	f(x,y), & x=y,\\
	0, & x\ne y.
	\end{cases}
	\end{align*}
	An element $f\in I(X,R)$ said to be \textit{diagonal} whenever $f=f_D$. The diagonal elements form a commutative subalgebra of $I(X,R)$, which we denote by $D(X,R)$. Invertible elements of $I(X,R)$ are exactly those $f\in I(X,R)$ such that $f(x,x)\in R^*$ for all $x\in X$ by \cite[Theorem 1.2.3]{SpDo}. The {\it Jacobson radical} of $I(X,R)$, denoted $J(I(X,R))$, consists of $f\in I(X,R)$ such that $f(x,x)\in J(R)$ for all $x\in X$ (see~\cite[Theorem 4.2.5]{SpDo}). In particular, if $F$ is a field, then $f\in J(I(X,F))$ if and only if $f_D=0$. In this case $I(X,F)=D(X,F)\oplus J(I(X,F))$ as vector spaces. The center of $I(X,R)$ coincides with $\{r\dl\mid r\in R\}$ by \cite[Corollary 1.3.15]{SpDo}, provided that $X$ is connected.
	
	When $X$ is finite, $I(X,R)$ is a free $R$-module over $\{e_{xy}\mid x\leq y\}$, where
	\begin{align*}
	e_{xy}(u,v)=
	\begin{cases}
	1, & (u,v)=(x,y),\\
	0, & (u,v)\ne(x,y).
	\end{cases}
	\end{align*} 
	The set $\{e_{xy}\mid x\leq y\}$ will be called \textit{the standard basis} of $I(X,R)$. Observe that $e_{xy}e_{uv}=\delta(y,u)e_{xv}$ and $f=\sum_{x\le y}f(x,y)e_{xy}$ for any $f\in I(X,R)$. If $R$ is a field, then $\B:=\{e_{xy}\mid x<y\}$ is a basis of $J(I(X,R))$. We will write $e_{x}:=e_{xx}$, and more generally, $e_A=\sum_{a\in A}e_{aa}$ for $A\sst X$. Then $\{e_x\mid x\in X\}$ is a set of pairwise orthogonal idempotents of $I(X,R)$ which form a basis of $D(X,R)$. If $R$ is indecomposable, i.e., has no idempotents except $0$ and $1$, then all idempotents $e_x$ are primitive.
	
	If $F$ is a field, then it is known~\cite{Stanley,Baclawski72,BruL,BFS12} that any (anti-)automorphism of $I(X,F)$ is the composition of an inner automorphism, the (anti-)automorphism $\widehat\lb$ induced by an order (anti-)automorphism $\lb$ of $X$ and the automorphism $M_\sg$ induced by a {\it multiplicative} element $\sg\in I(X,F)$. The latter means that $\sg(x,y)\in F^*$ for all $x\le y$ and $\sg(x,y)\sg(y,z)=\sg(x,z)$ for all $x\le y\le z$. The corresponding automorphism, acting as $M_\sg(f)(x,y)=\sg(x,y)f(x,y)$ for all $f\in I(X,F)$, is also called {\it multiplicative}.

	\section{Idempotent preservers on $I(X,R)$ in the $2$-torsion-free case}\label{sec-2-torsion-free}
	
	We start with a result whose proof mimics the proof of~\cite[Theorem 5]{Molnar-Semrl98} (see also~\cite[Theorem 2.1]{Bresar-Semrl93a}). 
	
	\begin{prop}\label{Idemp-pres-is-Jordan-homo}
		Let $X$ be finite and $R$ be $2$-torsion-free. Then any idempotent preserver on $I(X,R)$ with values in an associative $R$-algebra preserves the Jordan product.
	\end{prop}
	\begin{proof}
		Let $\vf$ be an idempotent preserver $I(X,R)\to B$, where $B$ is an associative algebra.
		We will verify \cref{vf(ab+ba)} on the elements of the standard basis of $I(X,R)$. Thus, let $f=e_{xy}$ and $g=e_{uv}$.
		
		\textit{Case 1.} Both $f$ and $g$ are diagonal.
		
		\textit{Case 1.1.} $x=y=u=v$. This case is trivial, because $f=g=e_x$ is an idempotent.
		
		\textit{Case 1.2.} $x=y\ne u=v$. Then $f$ and $g$ are orthogonal idempotents, so $f+g$ is an idempotent. It follows from $\vf(f+g)^2=\vf(f+g)$ that $\vf(f)\vf(g)+\vf(g)\vf(f)=0$.
		
		\textit{Case 2.} Exactly one of the elements $f$ and $g$ is diagonal.
		
		\textit{Case 2.1.} $x=y=u<v$. Then $fg=e_{xv}$ and $gf=0$. Observe that $e_x+e_{xv}$ is an idempotent, so
		\begin{align}\label{vf(e_x)vf(e_xv)+vf(e_xv)vf(e_x)+vf(e_xv)^2=vf(e_xv)}
		\vf(e_x)\vf(e_{xv})+\vf(e_{xv})\vf(e_x)+\vf(e_{xv})^2=\vf(e_{xv}).
		\end{align}
		Similarly, since $e_x-e_{xv}$ is an idempotent, we obtain
		\begin{align}\label{-vf(e_x)vf(e_xv)-vf(e_xv)vf(e_x)+vf(e_xv)^2=-vf(e_xv)}
		-\vf(e_x)\vf(e_{xv})-\vf(e_{xv})\vf(e_x)+\vf(e_{xv})^2=-\vf(e_{xv}).
		\end{align}
		Combining \cref{vf(e_x)vf(e_xv)+vf(e_xv)vf(e_x)+vf(e_xv)^2=vf(e_xv),-vf(e_x)vf(e_xv)-vf(e_xv)vf(e_x)+vf(e_xv)^2=-vf(e_xv)} and using the fact that $R$ is $2$-torsion-free, we conclude that $\vf(e_x)\vf(e_{xv})+\vf(e_{xv})\vf(e_x)=\vf(e_{xv})$, as desired. 
		
		\textit{Case 2.2.} $x<y=u=v$. This case is proved similarly to Case 2.1, using the idempotents $e_y+e_{xy}$ and $e_y-e_{xy}$.
		
		\textit{Case 2.3.} $x=y$, $u<v$ and $x\not\in\{u,v\}$. Then $fg=gf=0$. Since $e_x+e_u+e_{uv}$ is an idempotent, then $\vf(e_x+e_u+e_{uv})^2=\vf(e_x+e_u+e_{uv})$. Using the results of Cases 1.1, 1.2 and \cref{vf(e_x)vf(e_xv)+vf(e_xv)vf(e_x)+vf(e_xv)^2=vf(e_xv)}, we obtain $\vf(e_x)\vf(e_{uv})+\vf(e_{uv})\vf(e_x)=0$, as desired.
		
		\textit{Case 2.4.} $x<y$, $u=v$ and $u\not\in\{x,y\}$. This case follows from Case 2.3 by commutativity of the Jordan product.
		
		\textit{Case 3.} Neither $f$ nor $g$ is diagonal.
		
		\textit{Case 3.1.} $x=u$ and $y=v$. Then $f=g=e_{xy}$, where $f^2=0$. Taking the idempotent $e_x+e_{xy}$, we obtain $\vf(e_{xy})^2=0$ as in Case 2.1.
		
		\textit{Case 3.2.} $x=u$ and $y\ne v$. Then $fg=gf=0$. Consider the idempotent $e_x+e_{xy}+e_{xv}$. It follows from $\vf(e_x+e_{xy}+e_{xv})^2=\vf(e_x+e_{xy}+e_{xv})$ and Cases 1.1, 2.1 and 3.1 that $\vf(e_{xy})\vf(e_{xv})+\vf(e_{xv})\vf(e_{xy})=0$, as desired.
		
		\textit{Case 3.3.} $y=v$ and $x\ne u$. This case is proved similarly to Case 3.2, using the idempotent $e_y+e_{xy}+e_{uy}$ and Cases 1.1, 2.2 and 3.1.
		
		\textit{Case 3.4.} $y=u$. Then $fg=e_{xv}$ and $gf=0$. Observe that $e_y+e_{xy}+e_{yv}+e_{xv}$ is an idempotent. So, it follows from $\vf(e_y+e_{xy}+e_{yv}+e_{xv})^2=\vf(e_y+e_{xy}+e_{yv}+e_{xv})$ and Cases 1.1, 2.1, 2.2, 2.3, 3.2 and 3.3 that $\vf(e_{xy})\vf(e_{yv})+\vf(e_{yv})\vf(e_{xy})=\vf(e_{xv})$, as desired.
		
		\textit{Case 3.5.} $x=v$. This case follows from Case 3.4 by commutativity of the Jordan product.
		
		\textit{Case 3.6.} $\{x,y\}\cap\{u,v\}=\emptyset$. Then $fg=gf=0$. Using the idempotent $e_x+e_u+e_{xy}+e_{uv}$ and Cases 1.1, 1.2, 2.1, 2.3, we get $\vf(e_{xy})\vf(e_{uv})+\vf(e_{uv})\vf(e_{xy})=0$, as desired.
	\end{proof}
	
	\section{Idempotent preservers of $I(X,F)$, where $\ch(F)=2$ and $|F|>2$}\label{sec-char(F)=2-and-|F|>2}
	
	\begin{rem}\label{Idemp-pres-over-field-with>2-elem}
		The result of \cref{Idemp-pres-is-Jordan-homo} remains valid if $R$ is a field (of an arbitrary characteristic) with $|R|>2$. Indeed, it only suffices to prove Cases 2.1 and 2.2 which used the $2$-torsion-free condition. For Case 2.1 observe that $e_x+re_{xv}$ is an idempotent for all $r\in R$. Therefore,
		\begin{align}\label{vf(e_x+re_xy)^2=vf(e_x+re_xy)}
		r(\vf(e_x)\vf(e_{xv})+\vf(e_{xv})\vf(e_x)-\vf(e_{xv}))+r^2\vf(e_{xv})^2=0
		\end{align}
		for arbitrary $r\in R$. Choose two distinct $r_1,r_2\in R^*$ (this is possible because $|R|>2$). Writing \cref{vf(e_x+re_xy)^2=vf(e_x+re_xy)} for $r_1$ and $r_2$ we obtain a homogeneous system of two linear equations in $\vf(e_x)\vf(e_{xv})+\vf(e_{xv})\vf(e_x)-\vf(e_{xv})$ and $\vf(e_{xv})^2$. Its determinant is $r_1r_2^2-r_1^2r_2=r_1r_2(r_2-r_1)\ne 0$. Hence it has only the trivial solution. Case 2.2 is proved similarly.
	\end{rem}
	
	\begin{cor}\label{idemp-pres-is-Lie-auto}
		Let $X$ be a finite poset and $F$ be a field with $\ch(F)=2$ and $|F|>2$. Then any bijective idempotent preserver $\vf:I(X,F)\to I(X,F)$ is a Lie automorphism of $I(X,F)$, so, whenever $X$ is connected, $\vf$ admits the description given in~\cite{FKS}.
	\end{cor}
	
	We are going to describe those Lie automorphisms of $I(X,F)$ which preserve the idempotents of $I(X,F)$.
	
	\begin{lem}\label{vf-idemp-pres<=>vf(e_x)-idemp}
		Let $X$ be a finite connected poset and $F$ a field with $\ch(F)=2$. For a Lie automorphism $\vf$ of $I(X,F)$ the following are equivalent:
		\begin{enumerate}
			\item $\vf$ preserves the squares;\label{vf-pres-squares}
			\item $\vf$ preserves the idempotents;\label{vf-pres-idempotents}
			\item $\vf(e_x)\in E(I(X,F))$ for all $x\in X$.\label{vf(e_x)-idempotent}
		\end{enumerate}
	\end{lem}
	\begin{proof}
		The implications \cref{vf-pres-squares} $\impl$ \cref{vf-pres-idempotents} $\impl$ \cref{vf(e_x)-idempotent} are trivial, so it only remains to establish \cref{vf(e_x)-idempotent} $\impl$ \cref{vf-pres-squares}. Assume that $\vf$ satisfies \cref{vf(e_x)-idempotent}. In view of \cite[Theorem 4.15]{FKS} we may take $\vf$ to be elementary. Then $\vf(e_{xy})=\sg(x,y)\0(e_{xy})$ for all $x<y$, where $\0(e_{xy})\in\B$ and $\sg(x,y)\in F^*$ by \cite[Lemma 4.3]{FKS}. By \cref{vf(e_x)-idempotent} we have $\vf(e_x)^2=\vf(e_x)=\vf(e_x^2)$ for all $x\in X$. Moreover, $\vf(e_{xy})^2=\sg(x,y)^2\0(e_{xy})^2=0=\vf(0)=\vf(e_{xy}^2)$ for all $x<y$. Since for any $f\in I(X,F)$ one has
		\begin{align*}
		f^2=\sum_{x\le y}f(x,y)^2e_{xy}^2+\sum_{(x,y)\ne(u,v)}f(x,y)f(u,v)(e_{xy}\circ e_{uv})
		\end{align*}
		and $\vf$ preserves the Jordan product, it follows that $\vf(f^2)=\vf(f)^2$ as desired.
	\end{proof}
	
	\begin{lem}\label{vf(e_x_i)-idemp<=>c_i=0-or-1}
		Let $X=\{x_1,\dots,x_n\}$ be a finite connected poset, $F$ a field with $\ch(F)=2$, $\vf$ an elementary Lie automorphism of $I(X,F)$ and $\vf=\tau_{\0,\sg,c}$ its description as in \cite[Theorem 5.18]{FKS}. Then for all $1\le i\le n$ we have $\vf(e_{x_i})\in E(I(X,F))$ if and only if $c_i\in\{0,1\}$. 
	\end{lem}
	\begin{proof}
		Assume that $\vf(e_{x_i})\in E(I(X,F))$. Then by \cite[Definition 5.17]{FKS} we have $c_i=\vf(e_{x_i})(x_1,x_1)\in\{0,1\}$. Conversely, if $c_i\in\{0,1\}$, then $\vf(e_{x_i})(x_1,x_1)\in\{0,1\}$, and moreover by formulas (11) and (12) from \cite[Lemma 5.8]{FKS} we get $\vf(e_{x_i})(x_j,x_j)=\vf(e_{x_i})(x_1,x_1)+s(i,j)$, where $s(i,j)\in\{0,1\}$ because $\ch(F)=2$. Thus, $\vf(e_{x_i})(x_j,x_j)\in\{0,1\}$ for all $1\le j\le n$. Since $\vf(e_{x_i})\in D(I(X,F))$, it is an idempotent.
	\end{proof}
	
	\begin{thrm}\label{tau_0_sg_c-idemp-pres<=>c_i=0-or-1}
		Let $X$ be a finite poset and $F$ be a field with $\ch(F)=2$ and $|F|>2$. Then the bijective linear idempotent preservers $\vf:I(X,F)\to I(X,F)$ are exactly the Lie automorphisms of $I(X,F)$ of the form $\psi\circ\tau_{\0,\sg,c}$, where $\psi$ is an inner automorphism of $I(X,F)$ and $\tau_{\0,\sg,c}$ is an elementary Lie automorphism of $I(X,F)$ in which $c_i\in\{0,1\}$ for all $i$.
	\end{thrm}
	\begin{proof}
		A consequence of \cref{idemp-pres-is-Lie-auto,vf-idemp-pres<=>vf(e_x)-idemp,vf(e_x_i)-idemp<=>c_i=0-or-1}.
	\end{proof}
	
	The following bijective linear map was used in~\cite{FKS} as an example of a non-proper~\cite{FKS2} Lie automorphism (i.e., a Lie automorphism which is not the sum of an automorphism or the negative of an anti-automorphism and a central-valued linear map) of $I(X,F)$. It turns out to be an idempotent preserver of $I(X,F)$.
	\begin{exm}\label{rank-vf(e_x)-diff-from-1-and-n-1}
		Let $F$ be a field with $\ch(F)=2$ and $X=\{1,2,3,4\}$ with the following Hasse diagram. 
		
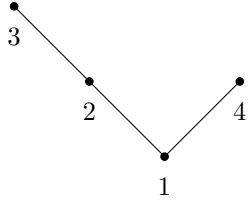
\begin{center}
			\begin{tikzpicture}
			\draw  (0,0)-- (-1,1);
			\draw  (-1,1)-- (-2,2);
			\draw  (0,0)-- (1,1);
			\draw [fill=black] (-1,1) circle (0.05);
			\draw  (-1,0.6) node {$2$};
			\draw [fill=black] (-2,2) circle (0.05);
			\draw  (-2,1.6) node {$3$};
			\draw [fill=black] (1,1) circle (0.05);
			\draw  (1,0.6) node {$4$};
			\draw [fill=black] (0,0) circle (0.05);
			\draw  (0,-0.4) node {$1$};
			\end{tikzpicture}
		\end{center}
		Consider the bijective linear map $\vf:I(X,F)\to I(X,F)$ acting on the natural basis as follows:
		$
		\vf(e_1) = e_3 + e_4,
		\vf(e_{12}) = e_{23},
		\vf(e_{13}) = e_{13},
		\vf(e_{14}) = e_{14},
		\vf(e_2) = e_1 + e_3 + e_4,
		\vf(e_{23}) = e_{12},
		\vf(e_3) = e_2 + e_3,
		\vf(e_4) = e_4.    
		$
		It was shown in~\cite[Example 5.20]{FKS} that $\vf$ is an elementary Lie automorphism of $I(X,F)$ of the form $\tau_{\0,\sg,c}$, where $c_1=c_3=c_4=0$ and $c_2=1$. Hence, $\vf$ is an idempotent-preserver of $I(X,F)$ by \cref{tau_0_sg_c-idemp-pres<=>c_i=0-or-1}. 
	\end{exm}
	
	
	An example of a proper Lie automorphism of $I(X,F)$ which is not an idempotent-preserver can also be easily constructed.
	\begin{exm}\label{Lie-auto-non-idemp-pres}
		Let $X=\{1,2\}$, $1<2$, and $F$ is a field with $\ch(F)=2$ and $|F|>2$. Define the linear map $\vf:I(X,F)\to I(X,F)$ as follows: $\vf(e_1)=e_1$, $\vf(e_{12})=e_{12}$ and $\vf(e_2)=e_2+r\dl$, where $r\not\in \{0,1\}$. Then $\vf$ is an elementary Lie automorphism of $I(X,F)$ of the form $\tau_{\0,\sg,c}$, where $c_1=1$ and $c_2=r$. Since $r\not\in \{0,1\}$, $\vf$ is not an idempotent-preserver by \cref{tau_0_sg_c-idemp-pres<=>c_i=0-or-1}. Observe moreover that $\vf=\id+\nu$, where $\nu(e_{12})=\nu(e_1)=0$ and $\nu(e_2)=r\dl$, so $\vf$ is proper~\cite{FKS2}.
	\end{exm}
	

	If $R=\mathbb{Z}_2$, then the statements of \cref{Idemp-pres-is-Jordan-homo,idemp-pres-is-Lie-auto} are not true.
	\begin{exm}
		Let $X=\{1,2\}$ with $1<2$. Consider $\vf: I(X,\Z_2)\to I(X,\Z_2)$ defined as follows: $\vf(f)=f+f(1,2)\dl$. It is clear that $\vf$ is bijective (with $\vf\m=\vf$) and additive. If $f^2=f$, then $\vf(f)^2=(f+f(1,2)\dl)^2=f^2+f(1,2)^2\dl^2=f+f(1,2)\dl$, so $\vf$ preserves idempotents. Observe that $\vf(e_1)=e_1$ and $\vf(e_{12})=e_{12}+\dl$, so $\vf(e_1)\circ\vf(e_{12})=e_1\circ (e_{12}+\dl)=e_1\circ e_{12}=e_{12}$. However, $\vf(e_1\circ e_{12})=\vf(e_{12})=e_{12}+\dl$, so $\vf$ does not preserve the Jordan product.
	\end{exm}

	\section{Jordan automorphisms of $I(X,F)$}\label{sec-Jordan-auto}
	
	In this section $F$ will be a field (of an arbitrary characteristic) and $X$ will be finite and connected (although some results hold for incidence algebras of arbitrary locally finite posets over commutative indecomposable unital rings). We recall that by a Jordan automorphism of an algebra $A$ we mean a bijective linear map $A\to A$ satisfying \cref{vf(a^2),vf(aba)}.
	
	\begin{lem}\label{vf(e_x)-conj-to-e_lb(x)}
		Let $\vf$ be a Jordan automorphism of $I(X,F)$. Then there exist an invertible element $\bt\in I(X,F)$ and a bijection $\lb:X\to X$ such that $\vf(e_x)=\bt e_{\lb(x)}\bt\m$ for all $x\in X$.
	\end{lem}
	\begin{proof}
		Clearly, $\vf(e_x)$ is an idempotent for all $x\in X$ by \cref{vf(a^2)}. Moreover, it is primitive. Indeed, assume that $\vf(e_x)=f+g$, where $f$ and $g$ are orthogonal idempotents. Since $\vf\m$ is also a Jordan automorphism, $\vf\m(f)$ and $\vf\m(g)$ are orthogonal idempotents by~\cite[Corollary 2]{Jacobson-Rickart50}. But $e_x$ is primitive, so $e_x=\vf\m(f)+\vf\m(g)$ implies $\vf\m(f)=0$ or $\vf\m(g)=0$. Thus, $f=0$ or $g=0$ as desired.
		
		By~\cite[Lemma 1]{Khripchenko-Novikov09} the idempotent $\vf(e_x)$ is conjugated to $e_{\lb(x)}$ for some $\lb(x)\in X$. If $x\ne y$, then $\vf(e_x)$ and $\vf(e_y)$ are orthogonal idempotents by~\cite[Corollary 2]{Jacobson-Rickart50}, so $\lb(x)\ne \lb(y)$ (otherwise $(\vf(e_x)\vf(e_y))(\lb(x),\lb(x))=1\ne 0$). Hence, $\lb$ is injective, and thus bijective because $X$ is finite.
		
		Define $\bt=\sum_{x\in X}\vf(e_x)e_{\lb(x)}$.  It is invertible, because $\bt(\lb(x),\lb(x))=1$ for all $x\in X$. Moreover, $\bt e_{\lb(x)}=\vf(e_x)e_{\lb(x)}$ due to orthogonality of $e_{\lb(x)}$ and $e_{\lb(y)}$ for $x\ne y$. Similarly, $\vf(e_x)\bt=\vf(e_x)e_{\lb(x)}$ in view of orthogonality of $\vf(e_x)$ and $\vf(e_y)$ for $x\ne y$. Thus, $\bt e_{\lb(x)}=\vf(e_x)\bt$, whence $\vf(e_x)=\bt e_{\lb(x)}\bt\m$.
	\end{proof}
	
	Composing $\vf$ with an inner automorphism of $I(X,F)$, we may assume that $\vf(e_x)=e_{\lb(x)}$ for some bijection $\lb:X\to X$.
	
	\begin{lem}\label{lb-auto-or-anti-auto}
		Let $\vf$ be a Jordan automorphism of $I(X,F)$ such that $\vf(e_x)=e_{\lb(x)}$ for some bijection $\lb:X\to X$. Then $\lb$ is either an automorphism or an anti-automorphism of $X$.
	\end{lem}
	\begin{proof}
		Let $x<y$. Observe that $e_{xy}=e_xe_{xy}e_y+e_ye_{xy}e_x$, so by \cref{vf(abc+cba)}
		\begin{align*}
		\vf(e_{xy})=e_{\lb(x)}\vf(e_{xy})e_{\lb(y)}+e_{\lb(y)}\vf(e_{xy})e_{\lb(x)}.
		\end{align*}
		Hence, either $\lb(x)<\lb(y)$, in which case 
		\begin{align}\label{vf(e_xy)=e_lb(x)vf(e_xy)e_lb(y)}
		\vf(e_{xy})=e_{\lb(x)}\vf(e_{xy})e_{\lb(y)}=\vf(e_{xy})(\lb(x),\lb(y))e_{\lb(x)\lb(y)},
		\end{align}
		or $\lb(y)<\lb(x)$, in which case 
		\begin{align}\label{vf(e_xy)=e_lb(y)vf(e_xy)e_lb(x)}
		\vf(e_{xy})=e_{\lb(y)}\vf(e_{xy})e_{\lb(x)}=\vf(e_{xy})(\lb(y),\lb(x))e_{\lb(y)\lb(x)}
		\end{align}
		(otherwise $\vf(e_{xy})=0$).
		
		Assume that $\lb(x)<\lb(y)$ and take another pair $u<v$. We are going to prove that $\lb(u)<\lb(v)$. Since $X$ is connected, there is a walk $y=x_1,\dots x_{m-1}=u$ from $y$ to $u$. We extend it to the sequence $x_0=x,x_1=y,x_2,\dots,x_{m-2},x_{m-1}=u,x_m=v$. We now prove by induction on $i$ that
		\begin{align*}
		x_i<x_{i+1}\impl\lb(x_i)<\lb(x_{i+1}),\text{ and }x_i>x_{i+1}\impl\lb(x_i)>\lb(x_{i+1})
		\end{align*}
		for all $0\le i\le m-1$. The base case holds by assumption. For the inductive step consider $4$ cases.
		
		\textit{Case 1.} $x_i<x_{i+1}<x_{i+2}$. By induction hypothesis $\lb(x_i)<\lb(x_{i+1})$. Suppose that $\lb(x_{i+1})>\lb(x_{i+2})$. Since $e_{x_ix_{i+2}}=e_{x_ix_{i+1}}\circ e_{x_{i+1}x_{i+2}}$, then by \cref{vf(ab+ba),vf(e_xy)=e_lb(x)vf(e_xy)e_lb(y),vf(e_xy)=e_lb(y)vf(e_xy)e_lb(x)}
		\begin{align*}
		\vf(e_{x_ix_{i+2}})&=\vf(e_{x_ix_{i+1}})\circ\vf(e_{x_{i+1}x_{i+2}})\\
		&=(e_{\lb(x_i)}\vf(e_{x_ix_{i+1}})e_{\lb(x_{i+1})})\circ (e_{\lb(x_{i+2})}\vf(e_{x_{i+1}x_{i+2}})e_{\lb(x_{i+1})})=0,
		\end{align*}
		a contradiction. Thus, $\lb(x_{i+1})<\lb(x_{i+2})$.
		
		\textit{Case 2.} $x_i<x_{i+1}>x_{i+2}$. By induction hypothesis $\lb(x_i)<\lb(x_{i+1})$. Suppose that $\lb(x_{i+1})<\lb(x_{i+2})$. Since $e_{x_ix_{i+1}}\circ e_{x_{i+2}x_{i+1}}=0$, then by \cref{vf(ab+ba),vf(e_xy)=e_lb(x)vf(e_xy)e_lb(y),vf(e_xy)=e_lb(y)vf(e_xy)e_lb(x)}
		\begin{align*}
		0&=\vf(e_{x_ix_{i+1}})\circ\vf(e_{x_{i+2}x_{i+1}})\\
		&=\vf(e_{x_ix_{i+1}})(\lb(x_i),\lb(x_{i+1}))\vf(e_{x_{i+2}x_{i+1}})(\lb(x_{i+1}),\lb(x_{i+2}))e_{\lb(x_i)\lb(x_{i+2})},
		\end{align*}
		so either $\vf(e_{x_ix_{i+1}})(\lb(x_i),\lb(x_{i+1}))=0$ or $\vf(e_{x_{i+2}x_{i+1}})(\lb(x_{i+1}),\lb(x_{i+2}))=0$, i.e. either $\vf(e_{x_ix_{i+1}})=0$ or $\vf(e_{x_{i+2}x_{i+1}})=0$, a contradiction.
		
		\textit{Case 3.} $x_i>x_{i+1}>x_{i+2}$. This case is proved similarly to Case 1 using $e_{x_{i+2}x_i}=e_{x_{i+2}x_{i+1}}\circ e_{x_{i+1}x_i}$.
		
		\textit{Case 4.} $x_i>x_{i+1}<x_{i+2}$. This case is proved similarly to Case 2 using $e_{x_{i+1}x_i}\circ e_{x_{i+1}x_{i+2}}=0$.
		
		In particular, for $i=m-1$ we have $\lb(x_{m-1})<\lb(x_m)$, i.e. $\lb(u)<\lb(v)$, as desired. So, $\lb$ is order-preserving. Observe that $\lb\m$ corresponds to $\vf\m$, and since $\lb(x)<\lb(y)$ and $\lb\m(\lb(x))<\lb\m(\lb(y))$, by the above argument $\lb\m$ is also order-preserving. Thus, $\lb$ is an automorphism of $X$.
		
		Similarly, one proves that $\lb(x)>\lb(y)$ implies $\lb(u)>\lb(v)$ for any other pair $u<v$ in $X$. In this case $\lb$ is an anti-automorphism of $X$.
	\end{proof}
	
	Replacing $\vf$ by $(\widehat\lb)\m\circ\vf$, we may assume that $\vf(e_x)=e_x$ for all $x\in X$.
	
	\begin{lem}\label{vf-mult-auto}
		Let $\vf$ be a Jordan automorphism of $I(X,F)$ such that $\vf(e_x)=e_x$ for all $x\in X$. Then $\vf$ is a multiplicative automorphism of $I(X,F)$.
	\end{lem}
	\begin{proof}
		Define $\sg(x,y)=\vf(e_{xy})(x,y)$ for all $x\le y$. Thus, $\sg(x,x)=1$ and $\vf(e_{xy})=\sg(x,y)e_{xy}$ for all $x<y$ in view of \cref{vf(e_xy)=e_lb(x)vf(e_xy)e_lb(y)}. It only remains to show that $\sg(x,z)=\sg(x,y)\sg(y,z)$, whenever $x\le y\le z$. This is trivial, if $x=y$ or $y=z$, so let $x<y<z$. It follows from $e_{xz}=e_{xy}\circ e_{yz}$ that
		\begin{align*}
		\sg(x,z)e_{xz}&=\vf(e_{xz})=\vf(e_{xy})\circ\vf(e_{yz})\\
		&=\sg(x,y)\sg(y,z)(e_{xy}\circ e_{yz})=\sg(x,y)\sg(y,z)e_{xz},
		\end{align*}
		so $\sg(x,z)=\sg(x,y)\sg(y,z)$, as desired.
	\end{proof}
	
	\begin{thrm}\label{Jord-auto-is-auto-or-anti-auto}
		Let $X$ be a finite connected poset and $F$ a field. Then any Jordan automorphism of $I(X,F)$ is either an automorphism or an anti-automorphism.
	\end{thrm}
	\begin{proof}
		This follows from \cref{vf(e_x)-conj-to-e_lb(x),lb-auto-or-anti-auto,vf-mult-auto}.
	\end{proof}
	
	As a consequence of \cref{Idemp-pres-is-Jordan-homo,Jord-auto-is-auto-or-anti-auto} we obtain the following.
	\begin{cor}\label{bijective-idemp-pres-is-auto-or-anti-auto}
		Let $X$ be a finite connected poset and $F$ a field with $\ch(F)\ne 2$. Then any bijective linear idempotent preserver $I(X,F)\to I(X,F)$ is either an automorphism or an anti-automorphism.
	\end{cor}

	\section{Idempotent preservers of $I(X,\mathbb{Z}_2)$}\label{sec-I(X_Z_2)}
	
	\subsection{Some general facts on idempotent preservers in characteristic $2$}
	\begin{rem}\label{e+f-idempotent}
		Let $A$ be an algebra over a ring $R$ with $\ch(R)=2$ and $e,f\in E(A)$. Then $e+f\in E(A)$ if and only if $ef=fe$.
		
		Indeed, $(e+f)^2=e^2+ef+fe+f^2=e+f+ef+fe$, and over a ring of characteristic $2$ anti-commutativity is the same as commutativity.
	\end{rem}
	
	\begin{lem}\label{vf(e)vf(f)=vf(f)vf(e)}
		Let $A$ and $B$ be algebras over a ring $R$ with $\ch(R)=2$ and $\vf:A\to B$ an additive idempotent preserver. Then $\vf$ maps commuting idempotents of $A$ to commuting idempotents of $B$.
	\end{lem}
	\begin{proof}
		Let $e$ and $f$ be commuting idempotents of $A$. Then $e+f\in E(A)$ by \cref{e+f-idempotent}. Hence, $\vf(e)+\vf(f)=\vf(e+f)\in E(B)$, so $\vf(e)$ and $\vf(f)$ commute by \cref{e+f-idempotent}.
	\end{proof}
	
	\begin{cor}\label{vf-maps-E(C_A(e))-to-E(C_B(vf(e)))}
		Let $A$ and $B$ be algebras over a ring $R$ with $\ch(R)=2$ and $\vf:A\to B$ an additive idempotent preserver. For any $e\in E(A)$ we have 
		\begin{align}\label{vf(E(C_A(e)))-sst-E(C_B(vf(e)))}
		\vf(E(C_A(e)))\sst E(C_B(\vf(e))).
		\end{align}
		Moreover, if $\vf$ is surjective and strong, then \cref{vf(E(C_A(e)))-sst-E(C_B(vf(e)))} becomes equality.
	\end{cor}
	\begin{proof}
		The inclusion \cref{vf(E(C_A(e)))-sst-E(C_B(vf(e)))} follows from \cref{vf(e)vf(f)=vf(f)vf(e)}. Assume that $f$ is surjective and strong. Let $b\in E(C_B(\vf(e)))$. By surjectivity there exists $a\in A$ such that $\vf(a)=b$, and by strongness $a\in E(A)$. Since $\vf(a)$ commutes with $\vf(e)$, then $\vf(a+e)=\vf(a)+\vf(e)\in E(B)$. Therefore, $a+e\in E(A)$, because $\vf$ is strong. By \cref{e+f-idempotent} the idempotents $a$ and $e$ commute, thus $a\in E(C_A(e))$.
	\end{proof}
	
	
	\subsection{Idempotents of $I(X,R)$ and their centralizers}
	
	\begin{lem}[Private conversation with M. Dugas and D. Herden]\label{diagonalization-idemp}
		Let $\{\af_i\}_{i=1}^n$ be a set of pairwise commuting idempotents of $I(X,R)$. Then there exists an invertible $\bt\in I(X,R)$ such that $\af_i=\bt(\af_i)_D\bt\m$ for all $1\le i\le n$.
	\end{lem}
	\begin{proof}
		For any $f\in E(I(X,R))$ introduce the following notation: $f^1:=f$ and $f^0=\dl-f$. Thus, $f^p\in E(I(X,R))$ and $f^p$ is orthogonal to $f^{1-p}$ for any $p\in\{0,1\}$. Moreover, if $f$ and $g$ are commuting idempotents, then $f^p$ and $g^q$ are also commuting idempotents for all $p,q\in\{0,1\}$.  
		
		Denote $\ve_i=(\af_i)_D$, $1\le i\le n$, and set
		\begin{align*}
		\bt=\sum_{p_1,\dots,p_n\in\{0,1\}}\af_1^{p_1}\dots\af_n^{p_n}\ve_1^{p_1}\dots\ve_n^{p_n}.
		\end{align*}
		Observe that 
		\begin{align*}
		\bt\ve_i=\sum_{p_1,\dots,p_n\in\{0,1\},\  p_i=1}\af_1^{p_1}\dots\af_n^{p_n}\ve_1^{p_1}\dots\ve_n^{p_n}=\af_i\bt
		\end{align*}
		for all $1\le i\le n$. It remains to prove that $\bt$ is invertible. Since $(\af_i^{p_i})_D=\ve_i^{p_i}$ for all $1\le i\le n$, we have
		\begin{align*}
		\bt_D&=\sum_{p_1,\dots,p_n\in\{0,1\}}\ve_1^{p_1}\dots\ve_n^{p_n}\\
		&=\sum_{p_2,\dots,p_n\in\{0,1\}}\ve_1\ve_2^{p_2}\dots\ve_n^{p_n}+\sum_{p_2,\dots,p_n\in\{0,1\}}(\dl-\ve_1)\ve_2^{p_2}\dots\ve_n^{p_n}\\
		&=\sum_{p_2,\dots,p_n\in\{0,1\}}(\ve_1+\dl-\ve_1)\ve_2^{p_2}\dots\ve_n^{p_n}=\sum_{p_2,\dots,p_n\in\{0,1\}}\ve_2^{p_2}\dots\ve_n^{p_n}.
		\end{align*}
		By the obvious induction $\bt_D=\dl$.
	\end{proof}
	
	\begin{lem}\label{C_I(e_x)-descr}
		Let $A\sst X$ and $f\in I(X,R)$. Then $f\in C_{I(X,R)}(e_A)\iff f(x,a)=f(a,y)=0$ for all $a\in A$ and $x,y\not\in A$ with $x<a<y$.
	\end{lem}
	\begin{proof}
		Let $x\le y$. Consider all the possible cases.
		
		{\it Case 1.} $x,y\in A$. Then $(e_Af)(x,y)=f(x,y)=(fe_A)(x,y)$.
		
		{\it Case 2.} $x,y\not\in A$. Then $(e_Af)(x,y)=0=(fe_A)(x,y)$.
		
		{\it Case 3.} $x\in A$ and $y\not\in A$. Then $(e_Af)(x,y)=f(x,y)$ and $(fe_A)(x,y)=0$.
		
		{\it Case 4.} $x\not\in A$ and $y\in A$. Then $(e_Af)(x,y)=0$ and $(fe_A)(x,y)=f(x,y)$.
		
		Thus, for $e_Af=fe_A$ it is necessary and sufficient that $f(x,y)=0$ whenever either $x\in A, y\not\in A$ or $x\not\in A, y\in A$.
	\end{proof}
	
	Given $A\sst X$, we set $A^2_<:=\{(x,y)\in A^2\mid x<y\}$.
	
	\begin{cor}\label{basis-of-C_I(e_A)}
		Let $X$ be finite and $A\sst X$. Then 
		\begin{align*}
		C_{I(X,R)}(e_A)&=\gen{e_{xy}\mid x,y\in A}\oplus\gen{e_{xy}\mid x,y\not\in A}\\
		&=D(X,R)\oplus\gen{e_{xy}\mid (x,y)\in A^2_<\sqcup (X\setminus A)^2_<}. 
		\end{align*}
		In particular, $C_{I(X,R)}(e_A)$ is a free $R$-module of rank $|X|+|A^2_<|+|(X\setminus A)^2_<|$.
	\end{cor}
	
	We will need a basis of $C_{I(X,R)}(e_A)$ consisting of idempotents.
	\begin{rem}
		If $X$ is finite, then $\{e_x\}_{x\in X}\sqcup\{e_x+e_{xy}\}_{x<y}\sst E(I(X,R))$ is a basis of $I(X,R)$.
	\end{rem}
	
	\begin{cor}\label{C_I(e_A)-basis}
		Let $X$ be finite and $A\sst X$. Then 
		$$
		C_{I(X,R)}(e_A)=D(X,R)\oplus\gen{e_x+e_{xy}\mid (x,y)\in A^2_<\sqcup (X\setminus A)^2_<}.
		$$
		In particular, $C_{I(X,R)}(e_A)=\gen{E(C_{I(X,R)}(e_A))}$.
	\end{cor}
	
	\begin{lem}\label{inter-C_I(e_A)=D+<e_xy>}
		Let $X$ be finite and $x<y$ from $X$. Then 
		\begin{align*}
		\bigcap_{A\supseteq\{x,y\}}C_{I(X,R)}(e_A)=D(X,R)\oplus\gen{e_{xy}}.
		\end{align*}
	\end{lem}
	\begin{proof}
		The inclusion $D(X,R)\sst C_{I(X,R)}(e_A)$ is obvious, and the inclusion $\gen{e_{xy}}\sst C_{I(X,R)}(e_A)$ for all $A\supseteq\{x,y\}$ follows from \cref{C_I(e_x)-descr}. Thus, $D(X,R)\oplus\gen{e_{xy}}\sst \bigcap_{A\supseteq\{x,y\}}C_{I(X,R)}(e_A)$. Conversely, let $f\in \bigcap_{A\supseteq\{x,y\}}C_{I(X,R)}(e_A)$. Take $A=\{x,y\}$. Since $f\in C_{I(X,R)}(e_A)$, then 
		\begin{align*}
		f=f_D+f(x,y)e_{xy}+\sum_{u<v:\ u,v\not\in\{x,y\}}f(u,v)e_{uv}
		\end{align*}
		by \cref{C_I(e_x)-descr}. If $X=\{x,y\}$, we are done. Otherwise, take $A=\{x,y,z\}$, where $z\not\in\{x,y\}$. Since $f$ and $f_D+f(x,y)e_{xy}$ commute with $e_A$, then $g:=\sum_{u<v:\ u,v\not\in\{x,y\}}f(u,v)e_{uv}$ commutes with $e_A=e_x+e_y+e_z$. But it also commutes with $e_x+e_y$, so $g\in C_{I(X,R)}(e_z)$. Hence, $f(u,v)=0$ for all $u<v$ such that $z\in\{u,v\}$ by \cref{C_I(e_x)-descr}. It follows that
		\begin{align*}
		f=f_D+f(x,y)e_{xy}+\sum_{u<v:\ u,v\not\in\{x,y,z\}}f(u,v)e_{uv}.
		\end{align*}
		Applying consecutively this argument we finally prove that $f=f_D+f(x,y)e_{xy}\in D(X,R)\oplus\gen{e_{xy}}$.
	\end{proof}
	
	\subsection{The description of idempotent preservers of $I(X,\mathbb{Z}_2)$}
	
	From now on we focus on the situation where $X$ is a finite connected poset and $\vf$ is a bijective additive (and thus $\Z_2$-linear) idempotent preserver from $I(X,\Z_2)$ to itself. Since $I(X,\mathbb{Z}_2)$ is finite and $\vf$ is injective, then $\vf(E(I(X,\mathbb{Z}_2)))\sst E(I(X,\mathbb{Z}_2))$ implies $\vf(E(I(X,\mathbb{Z}_2)))=E(I(X,\mathbb{Z}_2))$. Hence, $\vf$ is strong.
	
	\begin{rem}\label{vf=eta-circ-psi}
		By \cref{diagonalization-idemp,vf(e)vf(f)=vf(f)vf(e)} we have $\vf=\eta\circ\psi$, where $\eta$ is an inner automorphism of $I(X,\Z_2)$ and $\psi(e_x)$ is a diagonal idempotent for all $x\in X$.
	\end{rem}
	
	Replacing $\vf$ by $\psi$ if necessary, we will assume that $\vf(e_x)$ is a diagonal idempotent for all $x\in X$. It follows that $\vf(e_A)$ is a diagonal idempotent for all $A\sst X$, and since $E(D(X,\Z_2))=D(X,\Z_2)$, we have 
	\begin{align}\label{vf-maps-D-to-D}
	\vf(D(X,\Z_2))=D(X,\Z_2).    
	\end{align}
	Thus, there exists a bijection $\xi:2^X\to 2^X$ such that 
	\begin{align}\label{vf(e_A)=e_xi(A)}
	\vf(e_A)=e_{\xi(A)}    
	\end{align}
	for all $A\sst X$. We write $\xi(x)$ for $\xi(\{x\})$.
	
	\begin{lem}\label{vf(C_I(e_A))=C_I(vf(e_A))}
		For any $A\sst X$ we have $\vf(C_{I(X,\Z_2)}(e_A))=C_{I(X,\Z_2)}\left(e_{\xi(A)}\right)$.
	\end{lem}
	\begin{proof}
		Since $\vf$ is strong, $\vf(E(C_{I(X,\Z_2)}(e_A)))=E(C_{I(X,\Z_2)}(e_{\xi(A)}))$ thanks to \cref{vf-maps-E(C_A(e))-to-E(C_B(vf(e)))}. By additivity and bijectivity we have 
		$$
		\vf(\gen{E(C_{I(X,\Z_2)}(e_A))})=\gen{E(C_{I(X,\Z_2)}(e_{\xi(A)}))}.
		$$
		It remains to use \cref{C_I(e_A)-basis}.
	\end{proof}
	
	\begin{lem}\label{vf(e_xy)-in-inter-centralizers}
		There exists a bijection $\0=\0_\vf:\B\to\B$ such that for all $x<y$ we have $\vf(\gen{e_{xy}}\oplus D(X,\Z_2))=\gen{\0(e_{xy})}\oplus D(X,\Z_2)$.
	\end{lem}
	\begin{proof}
		By \cref{inter-C_I(e_A)=D+<e_xy>,vf(C_I(e_A))=C_I(vf(e_A)),basis-of-C_I(e_A),distr-direct-sum,inter-spans}
		\begin{align*}
		\vf(\gen{e_{xy}}\oplus D(X,\Z_2))&=\bigcap_{A\supseteq\{x,y\}}\vf(C_{I(X,\Z_2)}(e_A))
		=\bigcap_{A\supseteq\{x,y\}}C_{I(X,\Z_2)}(e_{\xi(A)})\\
		&=\bigcap_{A\supseteq\{x,y\}}\left(D(X,\Z_2)\oplus\gen{e_{uv}\mid (u,v)\in \xi(A)^2_<\sqcup (X\setminus \xi(A))^2_<}\right)\\
		&=D(X,\Z_2)\oplus\bigcap_{A\supseteq\{x,y\}}\gen{e_{uv}\mid (u,v)\in \xi(A)^2_<\sqcup (X\setminus \xi(A))^2_<}\\
		&=D(X,\Z_2)\oplus\gen{e_{uv}\ \Big\vert\ (u,v)\in \bigcap_{A\supseteq\{x,y\}}\left(\xi(A)^2_<\sqcup (X\setminus \xi(A))^2_<\right)}.
		\end{align*}
		But $\dim(\vf(\gen{e_{xy}}\oplus D(X,\Z_2)))=\dim(\gen{e_{xy}}\oplus D(X,\Z_2))=|X|+1$, so the intersection $\bigcap_{A\supseteq\{x,y\}}\left(\xi(A)^2_<\sqcup (X\setminus \xi(A))^2_<\right)$ consists of exactly one element $(u,v)$. We thus define $\0(e_{xy})=e_{uv}$. For the bijectivity of the map $\0:\B\to\B$ it suffices to prove that $\0$ is injective. Indeed, if $\0(e_{xy})=\0(e_{x'y'})=e_{uv}$, then $\vf(e_{xy})=e_{uv}+d$ and $\vf(e_{x'y'})=e_{uv}+d'$ for some $d,d'\in D(X,\Z_2)$. Therefore, $\vf(e_{xy}-e_{x'y'})=d-d'\in D(X,\Z_2)$, whence $e_{xy}-e_{x'y'}\in D(X,\Z_2)$ by \cref{vf-maps-D-to-D}. But this can happen only if $e_{xy}=e_{x'y'}$.
	\end{proof}
	
	Taking into account \cref{vf(e_xy)-in-inter-centralizers}, write 
	\begin{align}\label{vf(e_xy)=0(e_xy)+nu(e_xy)}
	\vf(e_{xy})=\0(e_{xy})+\nu(e_{xy}),    
	\end{align}
	where $\nu(e_{xy})\in D(X,\Z_2)$.
	
	\begin{rem}\label{e_uv-circ-d}
		Let $u<v$ and $d\in D(X,\Z_2)$. Then $d\circ e_{uv}=(d(u,u)+d(v,v))e_{uv}\in\{0,e_{uv}\}$.
	\end{rem}
	
	\begin{lem}\label{vf(e_xy)^2=nu(e_xy)}
		For all $x<y$ we have $\vf(e_{xy})^2=\nu(e_{xy})$.
	\end{lem}
	\begin{proof}
		Using \cref{e_uv-circ-d} we have $\vf(e_{xy})^2=(\0(e_{xy})+\nu(e_{xy}))^2=\0(e_{xy})^2+\0(e_{xy})\circ\nu(e_{xy})+\nu(e_{xy})^2=\0(e_{xy})\circ\nu(e_{xy})+\nu(e_{xy})\in\{\nu(e_{xy}),\vf(e_{xy})\}$. But $\vf(e_{xy})\not\in E(I(X,\Z_2))$, so $\vf(e_{xy})^2=\nu(e_{xy})$.
	\end{proof}
	
	\begin{cor}\label{nu(e_xy)-circ-0(e_xy)-is-zero}
		For all $x<y$ we have $\0(e_{xy})\in C_{I(X,\Z_2)}(\nu(e_{xy}))$.
	\end{cor}
	
	\begin{lem}\label{0(e_xy)=e_xi(x)-circ-0(e_xy)}
		For all $x<y$ we have $\0(e_{xy})=e_{\xi(x)}\circ\0(e_{xy})=e_{\xi(y)}\circ\0(e_{xy})$.
	\end{lem}
	\begin{proof}
		Since $e_x+e_{xy}\in E(I(X,\Z_2))$, we have
		\begin{align}\label{vf(e_xy)=vf(e_x)-circ-vf(e_xy)+vf(e_xy)^2}
		\vf(e_{xy})=\vf(e_x)\circ\vf(e_{xy})+\vf(e_{xy})^2.
		\end{align}
		By \cref{vf(e_A)=e_xi(A),vf(e_xy)=0(e_xy)+nu(e_xy),vf(e_xy)^2=nu(e_xy)} and the commutativity of $D(X,\Z_2)$, the right-hand side of \cref{vf(e_xy)=vf(e_x)-circ-vf(e_xy)+vf(e_xy)^2} equals $e_{\xi(x)}\circ\0(e_{xy})+\nu(e_{xy})$, whence $e_{\xi(x)}\circ\0(e_{xy})=\vf(e_{xy})-\nu(e_{xy})=\0(e_{xy})$. Similarly, $e_{\xi(y)}\circ\0(e_{xy})=\0(e_{xy})$ follows from the fact that $e_y+e_{xy}\in E(I(X,\Z_2))$.
	\end{proof}
	
	\begin{lem}\label{e_xi(z)-circ-0(e_xy)=0}
		For all $x<y$ and $z\not\in\{x,y\}$ we have $e_{\xi(z)}\circ\0(e_{xy})=0$.
	\end{lem}
	\begin{proof}
		Since $e_z,e_x+e_{xy}\in E(I(X,\Z_2))$ are orthogonal, then by \cref{vf(e_xy)=0(e_xy)+nu(e_xy),vf(e)vf(f)=vf(f)vf(e)} we have $0=\vf(e_z)\circ(\vf(e_x)+\vf(e_{xy}))=\vf(e_z)\circ\vf(e_{xy})=e_{\xi(z)}\circ(\0(e_{xy})+\nu(e_{xy}))=e_{\xi(z)}\circ\0(e_{xy})$.
	\end{proof}
	
	\begin{lem}\label{vf(e_xy)-circ-vf(e_xz)=0}
		For all $x<y,z$ we have $\vf(e_{xy})\circ\vf(e_{xz})=0$. Similarly for $x,y<z$ we have $\vf(e_{xz})\circ\vf(e_{yz})=0$.
	\end{lem}
	\begin{proof}
		Since $e_x+e_{xy}+e_{xz}\in E(I(X,\Z_2))$, then
		\begin{align}\label{vf(e_xy)+vf(e_xz)=vf(e_xy)^2+vf(e_xz)^2+...}
		\vf(e_{xy})+\vf(e_{xz})&=\vf(e_{xy})^2+\vf(e_{xz})^2+\vf(e_x)\circ\vf(e_{xy})+\vf(e_x)\circ\vf(e_{xz})\notag\\
		&\quad+\vf(e_{xy})\circ\vf(e_{xz}).
		\end{align}
		By \cref{vf(e_A)=e_xi(A),vf(e_xy)=0(e_xy)+nu(e_xy),0(e_xy)=e_xi(x)-circ-0(e_xy),vf(e_xy)^2=nu(e_xy)} the right-hand side of \cref{vf(e_xy)+vf(e_xz)=vf(e_xy)^2+vf(e_xz)^2+...} equals $\nu(e_{xy})+\nu(e_{xz})+\0(e_{xy})+\0(e_{xz})+\vf(e_{xy})\circ\vf(e_{xz})=\vf(e_{xy})+\vf(e_{xz})+\vf(e_{xy})\circ\vf(e_{xz})$, whence $\vf(e_{xy})\circ\vf(e_{xz})=0$. Similarly, using $e_z+e_{xz}+e_{yz}\in E(I(X,\Z_2))$, we prove that $\vf(e_{xz})\circ\vf(e_{yz})=0$.
	\end{proof}
	
	\begin{cor}\label{0(e_xy)-circ-0(e_xz)-is-zero}
		For all $x<y,z$ we have $\0(e_{xy})\circ\0(e_{xz})=0$ and $\nu(e_{xy})\circ\0(e_{xz})=0$. Similarly for $x,y<z$ we have $\0(e_{yz})\circ\0(e_{xz})=0$ and $\nu(e_{xz})\circ\0(e_{yz})=0$.
	\end{cor}
	\begin{proof}
		By \cref{vf(e_xy)=0(e_xy)+nu(e_xy),vf(e_xy)-circ-vf(e_xz)=0} we obtain $\nu(e_{xy})\circ\0(e_{xz})+\0(e_{xy})\circ\nu(e_{xz})+\0(e_{xy})\circ\0(e_{xz})=0$. Write $\0(e_{xy})=e_{uv}$ and $\0(e_{xz})=e_{pq}$. Then $ae_{uv}+be_{pq}+e_{uv}\circ e_{pq}=0$ for some $a,b\in\Z_2$ by \cref{e_uv-circ-d}. Observe that $e_{uv}\circ e_{pq}\in\B\setminus\{e_{uv},e_{pq}\}$, because $u<v$ and $p<q$. Moreover, $e_{uv}\ne e_{pq}$, because $\0$ is a bijection. Hence, $a=b=0$ and $e_{uv}\circ e_{pq}=0$ due to linear independence of $\B$, proving $\0(e_{xy})\circ\0(e_{xz})=0$ and $\nu(e_{xy})\circ\0(e_{xz})=\0(e_{xy})\circ\nu(e_{xz})=0$. The proof of $\0(e_{yz})\circ\0(e_{xz})=0$ and $\nu(e_{yz})\circ\0(e_{xz})=\0(e_{yz})\circ\nu(e_{xz})=0$ is similar.
	\end{proof}
	
	\begin{lem}\label{0(e_xy)-circ-0(e_yz)=0(e_xz)}
		For all $x<y<z$ we have $\0(e_{xz})=\0(e_{xy})\circ\0(e_{yz})$ and $\nu(e_{xy})\circ\0(e_{yz})=\0(e_{xy})\circ\nu(e_{yz})=0$.
	\end{lem}
	\begin{proof}
		Since $e_y+e_{xy}+e_{yz}+e_{xz}\in E(I(X,\Z_2))$, then
		\begin{align}\label{vf(e_xy)+vf(e_yz)+vf(e_xz)=vf(e_xy)^2+vf(e_yz)^2+vf(e_xz)^2}
		\vf(e_{xy})+\vf(e_{yz})+\vf(e_{xz})&=\vf(e_{xy})^2+\vf(e_{yz})^2+\vf(e_{xz})^2\notag\\
		&\quad+\vf(e_y)\circ\vf(e_{xy})+\vf(e_y)\circ\vf(e_{yz})+\vf(e_y)\circ\vf(e_{xz})\notag\\
		&\quad+\vf(e_{xy})\circ\vf(e_{yz})+\vf(e_{xy})\circ\vf(e_{xz})+\vf(e_{yz})\circ\vf(e_{xz}).
		\end{align}
		Now, using \cref{vf(e_A)=e_xi(A),vf(e_xy)=0(e_xy)+nu(e_xy),0(e_xy)=e_xi(x)-circ-0(e_xy),vf(e_xy)^2=nu(e_xy),vf(e_xy)-circ-vf(e_xz)=0,e_xi(z)-circ-0(e_xy)=0}, we see that the right-hand side of \cref{vf(e_xy)+vf(e_yz)+vf(e_xz)=vf(e_xy)^2+vf(e_yz)^2+vf(e_xz)^2} is $\vf(e_{xy})+\vf(e_{yz})+\nu(e_{xz})+\vf(e_{xy})\circ\vf(e_{yz})$, whence $\0(e_{xz})=\vf(e_{xy})\circ\vf(e_{yz})$. The latter equals $\nu(e_{xy})\circ\0(e_{yz})+\0(e_{xy})\circ\nu(e_{yz})+\0(e_{xy})\circ\0(e_{yz})$. Let $\0(e_{xy})=e_{pq}$, $\0(e_{yz})=e_{rs}$ and $\0(e_{xz})=e_{tu}$. In view of \cref{e_uv-circ-d} we have
		\begin{align*}
		e_{tu}=ae_{pq}+be_{rs}+e_{pq}\circ e_{rs}
		\end{align*}
		for some $a,b\in\Z_2$. Observe that $e_{pq}\circ e_{rs}\in\B\setminus\{e_{pq},e_{rs}\}$, because $p<q$ and $r<s$. Moreover, $e_{tu}\in\B\setminus\{e_{pq},e_{rs}\}$ because $\0$ is a bijection. Therefore, $a=b=0$ and $e_{tu}=e_{pq}\circ e_{rs}$ by linear independence of $\B$, proving the desired equalities.
	\end{proof}
	
	\begin{lem}\label{vf(e_xy)-circ-vf(e_uv)-is-zero}
		For all $x<y$ and $u<v$ with $\{x,y\}\cap\{u,v\}=\emptyset$ we have $\vf(e_{xy})\circ\vf(e_{uv})=0$.
	\end{lem}
	\begin{proof}
		The element $e_x+e_u+e_{xy}+e_{uv}$ is an idempotent, therefore,
		\begin{align}
		\vf(e_{xy})+\vf(e_{uv})&=\vf(e_{xy})^2+\vf(e_{uv})^2+\vf(e_x)\circ\vf(e_{xy})+\vf(e_x)\circ\vf(e_{uv})\notag\\
		&\quad+\vf(e_u)\circ\vf(e_{xy})+\vf(e_u)\circ\vf(e_{uv})+\vf(e_{xy})\circ\vf(e_{uv}).\label{(vf(e_x)+vf(e_u)+vf(e_xy)+vf(e_uv))^2}
		\end{align}
		By \cref{vf(e_xy)=0(e_xy)+nu(e_xy),vf(e_xy)^2=nu(e_xy),0(e_xy)=e_xi(x)-circ-0(e_xy),e_xi(z)-circ-0(e_xy)=0} the right-hand side of \cref{(vf(e_x)+vf(e_u)+vf(e_xy)+vf(e_uv))^2} equals $\vf(e_{xy})+\vf(e_{uv})+\vf(e_{xy})\circ\vf(e_{uv})$, whence $\vf(e_{xy})\circ\vf(e_{uv})=0$.
	\end{proof}
	
	\begin{cor}\label{0(e_xy)-circ-0(e_uv)-is-zero}
		For all $x<y$ and $u<v$ with $\{x,y\}\cap\{u,v\}=\emptyset$ we have $\0(e_{xy})\circ\0(e_{uv})=\nu(e_{xy})\circ\0(e_{uv})=\0(e_{xy})\circ\nu(e_{uv})=0$.
	\end{cor}
	\begin{proof}
		It follows from \cref{vf(e_xy)-circ-vf(e_uv)-is-zero} that $\nu(e_{xy})\circ\0(e_{uv})+\0(e_{xy})\circ\nu(e_{uv})+\0(e_{xy})\circ\0(e_{uv})=0$. Write $\0(e_{xy})=e_{pq}$ and $\0(e_{uv})=e_{rs}$. Then $ae_{pq}+be_{rs}+e_{pq}\circ e_{rs}=0$ for some $a,b\in\Z_2$. Since $e_{pq}\circ e_{rs}\in\B\setminus\{e_{pq},e_{rs}\}$ and $e_{pq}\ne e_{rs}$, we conclude that $a=b=0$ and $e_{pq}\circ e_{rs}=0$, as needed.
	\end{proof}
	
	\begin{prop}\label{0-preserves-circ}
		For all $x<y$ and $u<v$ we have 
		\begin{align}\label{0(e_xy-circ-e_uv)=0(e_xy)-circ-0(e_uv)}
		\0(e_{xy}\circ e_{uv})=\0(e_{xy})\circ\0(e_{uv}).
		\end{align}
	\end{prop}
	\begin{proof}
		If $e_{xy}\circ e_{uv}\ne 0$, then either $y=u$, in which case $e_{xy}\circ e_{uv}=e_{xv}$, or $x=v$, in which case $e_{xy}\circ e_{uv}=e_{uy}$. In both cases \cref{0(e_xy-circ-e_uv)=0(e_xy)-circ-0(e_uv)} follows from \cref{0(e_xy)-circ-0(e_yz)=0(e_xz)}. If $e_{xy}\circ e_{uv}=0$, then either $x=u$ or $y=v$, or $\{x,y\}\cap\{u,v\}=\emptyset$. In the first two cases we apply \cref{0(e_xy)-circ-0(e_xz)-is-zero}, and in the third case we apply \cref{0(e_xy)-circ-0(e_uv)-is-zero}. Thus, \cref{0(e_xy-circ-e_uv)=0(e_xy)-circ-0(e_uv)} holds for all $x<y$ and $u<v$.
	\end{proof}
	
	\begin{defn}
		A {\it shift map} is an additive map $\tau:I(X,\Z_2)\to I(X,\Z_2)$ such that $\tau(f)-f\in C(I(X,\Z_2))$ for all $f\in I(X,\Z_2)$. Clearly, $\tau$ is an idempotent preserver.
	\end{defn}
	
	\begin{prop}
		The additive map $\tau:I(X,\Z_2)\to I(X,\Z_2)$, such that 
		\begin{align}
		\tau|_{D(I(X,\Z_2))}&=\id_{D(I(X,\Z_2))},\label{tau|_D=identity}\\
		\tau(\0(e_{xy}))&=\0(e_{xy})+\nu(e_{xy}),\ x<y,\label{tau(e_xy)=e_xy+nu(e_xy)}
		\end{align}
		is a bijective shift map.
	\end{prop}
	\begin{proof}
		To show that $\tau$ is a shift map, it is enough to prove that $\nu(e_{xy})\in C(I(X,\Z_2))$ for all $x<y$. To this end, it suffices to make sure that $\nu(e_{xy})$ commutes with $\0(e_{uv})$ for all $u<v$, because $\0(\B)=\B$. Indeed, $\nu(e_{xy})$ commutes with $\0(e_{xy})$ by \cref{nu(e_xy)-circ-0(e_xy)-is-zero}, with $\0(e_{xv})$ and $\0(e_{uy})$ by \cref{0(e_xy)-circ-0(e_xz)-is-zero} and with $\0(e_{uv})$, such that $\{x,y\}\cap\{u,v\}=\emptyset$, by \cref{0(e_xy)-circ-0(e_uv)-is-zero}. The bijectivity of $\tau$ trivially follows from $\tau\circ\tau=\id$.
	\end{proof}
	
	\begin{cor}\label{vf-is-tau-psi}
		The map $\vf$ decomposes as $\tau\circ\psi$, where $\tau$ is given by \cref{tau|_D=identity,tau(e_xy)=e_xy+nu(e_xy)} and $\psi$ is a bijective additive idempotent preserver of $I(X,\Z_2)$.
	\end{cor}
	\begin{proof}
		Let us define $\psi=\tau\circ\vf$. Then $\psi$ is a bijective additive idempotent preserver of $I(X,\Z_2)$ as a composition of bijective additive idempotent preservers of $I(X,\Z_2)$. Moreover, since $\tau\circ\tau=\id$, we have $\vf=\tau\circ\psi$.
	\end{proof}
	
	In view of \cref{vf-is-tau-psi}, replacing $\vf$ by $\psi$ if necessary, we will assume that $\vf(e_{xy})=\0(e_{xy})$ for all $x<y$. 
	
	\begin{prop}\label{vf-Lie-auto}
		The map $\vf$ is a Lie automorphism of $I(X,\Z_2)$.
	\end{prop}
	\begin{proof}
		It is enough to prove that $\vf(e_{xy}\circ e_{uv})=\vf(e_{xy})\circ\vf(e_{uv})$ for all $x\le y$ and $u\le v$. If $x<y$ and $u<v$, then $e_{xy},e_{uv}\in\B$, so the result follows from \cref{0-preserves-circ}. If $x=y$ and $u<v$, then either $x\in\{u,v\}$, in which case $e_{xy}\circ e_{uv}=e_{uv}$, or $x\not\in\{u,v\}$, in which case $e_{xy}\circ e_{uv}=0$. Hence, the result follows from \cref{0(e_xy)=e_xi(x)-circ-0(e_xy),e_xi(z)-circ-0(e_xy)=0}. Finally, the case $x=y$ and $u=v$ is trivial by \cref{vf(e)vf(f)=vf(f)vf(e)} because $e_{x}$ and $e_{u}$ are commuting idempotents.
	\end{proof}
	
	\begin{thrm}\label{idemp-pres-of-I(X_Z_2)}
		Let $\vf:I(X,\Z_2)\to I(X,\Z_2)$ be a bijective additive map. Then $\vf$ is an idempotent preserver if and only if it is a composition of a bijective shift map and a Lie automorphism of $I(X,\Z_2)$.
	\end{thrm}
	\begin{proof}
		{\it The ``if'' part.} Since a shift map always preserves idempotents, it is enough to prove that a Lie automorphism preserves idempotents. Let $\vf$ be a Lie automorphism of $I(X,\Z_2)$. In view of \cite[Theorem 4.15]{FKS} we may assume that $\vf$ is elementary. Then $\vf=\tau_{\0,\sg,c}$ by \cite[Theorem 5.18]{FKS}. Since $c_i\in\Z_2=\{0,1\}$ for all $i$, it follows from \cref{vf-idemp-pres<=>vf(e_x)-idemp,vf(e_x_i)-idemp<=>c_i=0-or-1} that $\vf$ is an idempotent preserver.
		
		{\it The ``only if'' part.} Let $\vf$ be an idempotent preserver. In view of \cref{vf=eta-circ-psi,vf-is-tau-psi,vf-Lie-auto} there is a decomposition $\vf=\eta\circ\tau\circ\psi$, where $\eta$ is an inner automorphism, $\tau$ is a shift map and $\psi$ is a Lie automorphism. Set $\nu(f)=\tau(f)-f$. Then $\eta\circ\tau=\tau'\circ\eta$, where $\tau'(f)=f+\nu(\eta\m(f))$ for all $f\in I(X,\Z_2)$. Clearly, $\tau'$ is a shift map because $\nu(\eta\m(f))\in C(I(X,\Z_2))$. Thus, $\vf=\tau'\circ\psi'$, where $\psi'=\eta\circ\psi$ is a Lie automorphism of $I(X,\Z_2)$.
	\end{proof}

	\section{Tripotent preservers of $I(X,F)$}\label{sec-trip-pres}

	\begin{lem}\label{l maps orth trip}
		Let $\vf:R\to S$ be an additive tripotent preserver, where $R$ and $S$ are rings and $S$ is $n$-torsion-free with $n\in\{2,3\}$. Then $\vf$ maps orthogonal tripotents to orthogonal tripotents.
	\end{lem}
	\begin{proof}
		Let $a,b\in R$ be orthogonal tripotents. Then $a+b$ and $a-b$ are tripotents. It follows from $\vf(a+b)^3=\vf(a+b)$ that
		\begin{align}\label{vf(a+b)^3=vf(a+b)}
		0&=\vf(a)\vf(b)^2+\vf(b)^2\vf(a)+\vf(a)^2\vf(b)+\vf(b)\vf(a)^2\notag\\
		&\quad+\vf(a)\vf(b)\vf(a)+\vf(b)\vf(a)\vf(b).
		\end{align}
		Similarly, $\vf(a-b)^3=\vf(a-b)$ implies
		\begin{align}\label{vf(a-b)^3=vf(a-b)}
		0&=\vf(a)\vf(b)^2+\vf(b)^2\vf(a)-\vf(a)^2\vf(b)-\vf(b)\vf(a)^2\notag\\
		&\quad-\vf(a)\vf(b)\vf(a)+\vf(b)\vf(a)\vf(b).
		\end{align}
		Since $S$ is $2$-torsion-free, equalities \cref{vf(a+b)^3=vf(a+b),vf(a-b)^3=vf(a-b)} yield
		\begin{align}\label{vf(a)^2vf(b)+vf(b)vf(a)^2+vf(a)vf(b)vf(a)=0}
		\vf(a)^2\vf(b)+\vf(b)\vf(a)^2+\vf(a)\vf(b)\vf(a)=0.
		\end{align}
		It follows from \cref{vf(a)^2vf(b)+vf(b)vf(a)^2+vf(a)vf(b)vf(a)=0} that 
		\begin{align*}
		\vf(a)\vf(b)-\vf(b)\vf(a)&=\vf(a)^3\vf(b)-\vf(b)\vf(a)^3\\
		&=\vf(a)(\vf(a)^2\vf(b)+\vf(b)\vf(a)^2+\vf(a)\vf(b)\vf(a))\\
		&\quad-(\vf(a)^2\vf(b)+\vf(b)\vf(a)^2+\vf(a)\vf(b)\vf(a))\vf(a)=0.
		\end{align*}
		Therefore, \cref{vf(a)^2vf(b)+vf(b)vf(a)^2+vf(a)vf(b)vf(a)=0} becomes $3\vf(a)^2\vf(b)=0$, whence $\vf(a)^2\vf(b)=0$ because $S$ is $3$-torsion-free. Thus, $\vf(a)\vf(b)=\vf(a)^3\vf(b)=\vf(a)\cdot\vf(a)^2\vf(b)=0$. Obviously, $\vf(b)\vf(a)=0$ as well, because $\vf(a)$ and $\vf(b)$ commute. 
	\end{proof}
	
	
	The following definition is inspired by  the Peirce spaces associated to a tripotent in a Jordan triple (see for instance \cite[Definition 1.2.37]{Chu_Book}). These play an important role in \cite{BFPGMP08}.
	
	\begin{defn}
		Let $A$ be an associative algebra and $u\in A$ a $k$-potent with $k\ge 3$. Then $u^{k-1}\in E(A)$, so $u^{k-1}Au^{k-1}$ is a unital associative algebra under the same multiplication. We define the new associative product $\bullet_u$ on $u^{k-1}Au^{k-1}$ as follows $a\bullet_u b=au^{k-2}b$. Clearly $u$ is the identity element with respect to $\bullet_u$.
	\end{defn}

	\begin{prop} \label{p trip pres is triple homo}
		Let $X$ be finite and $R$ be $n$-torsion-free with $n\in\{2,3\}$. Then any $R$-linear tripotent preserver $\varphi:I(X,R)\to B$, where $B$ is an associative algebra, is a Jordan triple homomorphism.
	\end{prop}
	
	\begin{proof} 
		Let $e\in I(X,R)$ be an idempotent. Then $e$ and $\dl-e$ are orthogonal idempotents. By \cref{l maps orth trip}, $\varphi(e)$ and $\varphi(\dl-e)=\varphi(\dl)-\varphi(e)$ are orthogonal tripotents. The following equalities easily follow 
		\begin{align}
		\varphi(\dl)\varphi(e)&=\varphi(e)^2=\varphi(e)\varphi(\dl),\notag\\
		\varphi(e)^2\varphi(\dl)&=\varphi(e)\varphi(\dl)\varphi(e)=\varphi(\dl)\varphi(e)\varphi(\dl)=\varphi(\dl)\varphi(e)^2=\varphi(e),\label{eq prop trip pres 2}\\
		\varphi(\dl)^2\varphi(e)&=\varphi(e)=\varphi(e)\varphi(\dl)^2.\label{eq prop trip pres 3}
		\end{align}

		
		Let us write $u=\vf(\dl)$ and $p=\vf(\dl)^2.$ Observe that $u$ is a tripotent, and hence $p$ is an idempotent. Equality \cref{eq prop trip pres 3} shows that $p$ commutes with $\varphi(e_x)$ for all $x\in X$. We shall show that $p$ also commutes with $\varphi(e_{xy})$ for all $x<y$. Applying \cref{eq prop trip pres 3} to the idempotent $e_x+ e_{xy}$, we have
		$$
		p\varphi(e_x)+p\varphi(e_{xy})=\varphi(e_x)+\varphi(e_{xy})= \varphi(e_x)p+\varphi(e_{xy})p.
		$$
		The above equalities together with \cref{eq prop trip pres 3} in which $e=e_x$ yield
		$$ 
		p\varphi(e_{xy})=\varphi(e_{xy})=\varphi(e_{xy})p,
		$$ thus proving the claim. Since $\{e_{xy}\mid x\leq y\}$ is a basis of $I(X,R)$, we have
		$$ p \varphi(f)=\varphi(f)=\varphi(f)p$$
		for all $f\in I(X,R)$. It follows that
		\begin{align*}
		\varphi(I(X,R))\subseteq pBp.
		\end{align*}
		
		Using $u\varphi(e)u=\varphi(e)$, which is a part of \cref{eq prop trip pres 2}, with $e=e_x$ and $e=e_x+e_{xy}$, we conclude that $u\varphi(e_{xy})u=\varphi(e_{xy})$. Thus,
		\begin{align}\label{vf(dl)vf(f)vf(dl)=vf(a)}
		u\varphi(f)u=\varphi(f)    
		\end{align}
		holds for every $f\in I(X,R).$
		
		Now,  $pBp$ endowed with the product $a\bullet_u b$ is an associative algebra with identity element $u$. Let $e$ be an idempotent in $I(X,R)$. Then $\varphi(e)\bullet_u \varphi(e)=\varphi(e)$ by \cref{eq prop trip pres 2}, so $\vf$, as a linear map 
		$I(X,R)\to (pBp,\bullet_u)$, preserves idempotents. It follows by \cref{Idemp-pres-is-Jordan-homo} that $\varphi:I(X,R)\to (pBp,\bullet_u)$ is a Jordan homomorphism. 
		
		We shall prove now that $\varphi:I(X,R)\to B$ is in fact a Jordan triple homomorphism.  Take $f,g\in I(X,R)$. Since $\vf(g)=u\varphi(g)u$ by \cref{vf(dl)vf(f)vf(dl)=vf(a)}, we have
		\begin{align*}
		\varphi(f)\varphi(g)\varphi(f)= \varphi(f)u\varphi(g)u\varphi(f)=\varphi(f)\bullet_u \varphi(g)\bullet_u \varphi(f).
		\end{align*}
		But $\varphi:I(X,R)\to (pBp,\bullet_u)$ is a Jordan homomorphism, so it is also a Jordan triple homomorphism $I(X,R)\to (pBp,\bullet_u)$. Therefore,
		\begin{align*}
		\varphi(f)\bullet_u \varphi(g)\bullet_u \varphi(f)=\varphi(fgf),   
		\end{align*}
		which completes the proof. 
	\end{proof}
	
	
	
	
	\begin{lem} \label{l T(1) unitary}
		Let $\varphi: A\to B$ be a Jordan triple isomorphism between unital associative $R$-algebras. Then $u=\vf(1)$ is a central (invertible) element of $B$ whose square is $1$ and $\vf:A \to (B,\bullet_u)$ preserves the Jordan product. 
	\end{lem}
	
	\begin{proof}
		We first show that  $u^2=1$. Obviously, $u$ is a tripotent because Jordan triple homomorphisms preserve tripotents. Choose $a\in A$ such that $\vf(a)=1$. By \cref{vf(aba)} we have
		\begin{align}\label{eq unitary 1}
		u\vf(a)u=\vf(1 \cdot a \cdot 1)=\vf(a)=1. 
		\end{align}
		Then multiplying $u^3=u$ by $u\vf(a)$ on the left and using \cref{eq unitary 1}, we obtain $u^2=1$, as desired. 
		
		Now let $b\in A.$ Using \cref{vf(aba)} we have $u\vf(b)u=\vf(b).$ We now multiply this by $u$ on the left to obtain $\vf(b)u=u\vf(b)$, that is, $u\in C(B)$, thus proving the claim.
		
		It remains to prove that $\vf:A\to (B,\bullet_u)$ preserves the Jordan product. Observe that this makes sense, because $u^2=1$, so $u^2Bu^2=B$. Since $a\circ b=a\cdot 1\cdot b+b\cdot 1\cdot a$, using \cref{vf(abc+cba)} we have 
		$$ 
		\vf(a\circ b)=\vf(a\cdot 1\cdot b+b\cdot 1\cdot a)=\vf(a)u\vf(b)+\vf(b)u\vf(a)=\vf(a)\bullet_u \vf(b)+\vf(b)\bullet_u \vf(a).
		$$ 
	\end{proof}
	
	\begin{prop}\label{Jordan-triple-auto-is-u.psi}
		Let $X$ be a finite connected poset and $F$ a field of characteristic different from $2$ and $3$. Then any Jordan triple automorphism of $I(X,F)$ is  of the form $r\psi$, where $r\in \{-1,1\}$ and $\psi$ is either an automorphism or an anti-automorphism.
	\end{prop}
	\begin{proof}
		\cref{l T(1) unitary} assures that $u=\vf(1)\in C(I(X,F))$ and $u^2=\delta$. Thus $u$ is of the form $u=r \delta$ for some $r\in F$ such that $r^2=1$. Since $F$ is a field, $r=\pm 1,$ that is, $u=\pm \delta.$ 
		If $u=\delta,$ then the product $a \bullet_u b$ coincides with the original product in $B$ and hence $u\vf=\vf$ is a Jordan automorphism by \cref{l T(1) unitary}. On the other hand, if $u=-\delta$ then $a \bullet_u b=-a b$, so $\vf$ is the negative of a Jordan automorphism in this case. 
		
		Thus, we have shown that $\varphi=u \psi, $ where $u\in \{-1,1\}$ and $\psi$ is a Jordan automorphism. By \cref{Jord-auto-is-auto-or-anti-auto}, $\psi$ is either an automorphism or an anti-automorphism.
	\end{proof}
	
	\begin{thrm}\label{trip-pres-description}
		Let $X$ be a finite connected poset and $F$ a field of characteristic different from $2$ and $3$. Then any bijective $F$-linear tripotent preserver of $I(X,F)$ is of the form $r\psi$, where $r\in \{-1,1\}$ and $\psi$ is either an automorphism or an anti-automorphism of $I(X,F)$. 
	\end{thrm}
	\begin{proof}
		A consequence of \cref{p trip pres is triple homo,Jordan-triple-auto-is-u.psi}.
	\end{proof}
	
	\section{$k$-potent preservers of $I(X,F)$}\label{sec-k-potent-pres}
	
	Let $k$ be a fixed integer greater than or equal to $3$.
	In order to prove unified versions of some results of \cref{sec-trip-pres} that hold for $k$-potent preservers, we need to restrict ourselves to the case of $F$-algebras, where $F$ is a field with a certain condition which will be specified below.
	
	\begin{rem}\label{rem p=0}
		Let $F$ be a field with $|F|\ge k$. Given $k$ distinct elements $t_0,\dots,t_{k-1}\in F$ and a polynomial $P=a_0 +a_1 t+\ldots +a_{k-1} t^{k-1}$ with coefficients in a vector space over $F$, we have $P(t)=0$ for all $t\in\{t_0,\dots,t_{k-1}\}$ if and only if $a_i=0$ for all $i=0,\dots,k-1$. Indeed, the matrix of the linear system $P(t_i)=0$, $i=0,\dots,k-1$, in variables $a_0,\dots,a_{k-1}$, is a Vandermonde matrix whose determinant is $\prod_{0\le i<j\le k-1}(t_j-t_i)\ne 0$, hence the system has only the trivial solution.  
	\end{rem}  
	
	The following  result is inspired by~\cite[Lemma~5]{Chan-Lim92} about $k$-th powers preservers on matrix algebras. It turns out that it is possible to relax the $k$-th powers preserving condition by imposing an extra assumption on the field. Observe that the case $k=3$ holds for rings (see \cref{l maps orth trip}). We include the proof for completeness.
	
	\begin{lem}\label{l maps orth if k-potents}
		Let $F$ be a field admitting a $(k-1)$-th primitive root of the unity. Then any linear $k$-potent preserver $\vf:A\to B$ between associative $F$-algebras maps orthogonal $k$-potents to commuting $k$-potents. If we further assume that $\ch(F) \nmid k$ then $\vf$  maps orthogonal $k$-potents to orthogonal $k$-potents.
	\end{lem}
	\begin{proof}
		Since $F$ contains a $(k-1)$-th primitive root of the unity, we can find $k-1$ distinct $(k-1)$-th roots of unity (equivalently, non-zero $k$-potents) $t_1,\ldots,t_{k-1}$ in $F.$  Fix  two orthogonal $k$-potents $e,f$ in $A$. It is easy to see that $e+t f$ is also a $k$-potent for all $t\in \{t_1,\ldots,t_{k-1}\}$. Thus for every $t\in \{t_1,\ldots,t_{k-1}\}$ we have 
		\begin{align}\label{varphi(e+t f)^k}
		\varphi(e)+t \varphi(f)&=\varphi(e+tf)=\varphi(e+tf)^k=(\varphi(e)+t \varphi(f))^k\notag\\
		&=
		\varphi(e)+ t a_1+\ldots + t^{k-1}a_{k-1}+t \varphi(f)
		\end{align} 
		for some $a_1,\ldots,a_{k-1} \in A.$ It follows from \cref{varphi(e+t f)^k} that $$ 
		P(t)=t a_1+\ldots + t^{k-1}a_{k-1}=0
		$$ 
		for all $t\in \{t_1,\ldots,t_{k-1}\}.$ Obviously, $P(0)$ is also $0$.  Using \cref{rem p=0} with $t_0=0,t_1,\dots,t_{k-1}$, we conclude that $a_i=0$ for all $i=1,\ldots,k-1.$ In particular, 
		\begin{align}\label{a1=0}
		0=a_1=\sum_{i=0}^{k-1}\varphi(e)^i\varphi(f) \varphi(e)^{k-i-1}. 
		\end{align}
		Multiplying \cref{a1=0} by $\vf(e)$ on the left we arrive at
		\begin{align}\label{a1vf(e)}
		0&=\vf(e)a_1=\varphi(e) \varphi(f) \varphi(e)^{k-1}+\varphi(e)^2 \varphi(f) \varphi(e)^{k-2}+\ldots\notag\\
		&\quad+\varphi(e)^{k-1} \varphi(f) \varphi(e)+\varphi(e)\vf(f).
		\end{align}
		Similarly multiplying \cref{a1=0} by $\vf(e)$ and on the right, we obtain
		\begin{align}\label{vf(e)a1}
		0&= a_1\vf(e)=\varphi(f) \varphi(e)+\varphi(e) \varphi(f) \varphi(e)^{k-1}+\ldots\notag\\
		&\quad+\varphi(e)^{k-2} \varphi(f) \varphi(e)^2+\varphi(e)^{k-1} \varphi(f) \varphi(e).
		\end{align}
		Now \cref{a1vf(e)} together with \cref{vf(e)a1} give 
		\begin{align*} 
		0=\varphi(e) a_1-a_1 \varphi(e)=\varphi(e) \varphi(f)-\varphi(f)\varphi(e),
		\end{align*}
		that is, $\vf(e)$ and $\vf(f)$ commute. Thus, \cref{a1=0} becomes $0=a_1=k \varphi(e)^{k-1}\varphi(f). $ If we assume that $\ch(F) \nmid k$, we can assure that $\varphi(e)^{k-1}\varphi(f)=0.$ Multiplying the latter equality on the left by $\varphi(e)$ we have $\varphi(e)\varphi(f)=0.$ Finally, since $\varphi(e)$ and $\varphi(f)$ commute we also have $\varphi(f) \varphi(e)=0.$
	\end{proof}
	
	
	
	Throughout the rest of the section $F$ will be a field admitting a $(k-1)$-th primitive root of the unity and such that $\ch(F) \nmid k$. 
	
	\begin{lem} \label{l k-potent pres 1}
		Let $X$ be a finite poset, $\varphi:I(X,F)\to B$ a linear $k$-potent preserver and $u=\vf(\dl)$. Then $u\in C_B(\vf(I(X,F)))$ and $\varphi: I(X,F)\to (u^{k-1}Bu^{k-1},\bullet_u)$ is a Jordan homomorphism.
	\end{lem}
	\begin{proof}
		Let $e\in E(I(X,F))$. Since $e$ and $\dl-e$ are orthogonal idempotents (in particular, $k$-potents), then $\vf(e)$ and $\vf(\dl-e)$ are orthogonal $k$-potents by \cref{l maps orth if k-potents}, so
		\begin{align}\label{vf(1)^n=vf(e)^n+vf(1-e)^n}
		u^n=\vf(e)^n+\vf(\dl-e)^n   
		\end{align}
		for all $n>0$. Multiplying \cref{vf(1)^n=vf(e)^n+vf(1-e)^n} with $n=k-1$ by $\vf(e)$ on the left or on the right, we obtain
		\begin{align}\label{vf(dl)^(k-1)vf(e)=vf(e)=vf(evf(dl)^(k-1))} 
		u^{k-1}\vf(e)=\vf(e)=\vf(e)u^{k-1}.
		\end{align}
		Applying \cref{vf(dl)^(k-1)vf(e)=vf(e)=vf(evf(dl)^(k-1))} to the idempotents $e_x$ and $e_x+e_{xy}$, $x<y$, we show that
		\begin{align}\label{u^(k-1)vf(e_xy)=vf(e_xy)=vf(e)u^(k-1)}
		u^{k-1}\vf(e_{xy})=\vf(e_{xy})=\vf(e_{xy})u^{k-1}
		\end{align}
		for all  $x\le y$. In particular,
		$$
		u^{k-1}\in C_B(\vf(I(X,F))) \mbox{ and } \vf( I(X,F))\subseteq u^{k-1} Bu^{k-1}.
		$$
		Now, for any $e\in E(I(X,F))$ using \cref{vf(1)^n=vf(e)^n+vf(1-e)^n} we have 
		$$
		\vf(e)\bullet_u \vf(e)=\vf(e)u^{k-2} \vf(e)=\vf(e)(\vf(e)^{k-2}+\vf(\dl-e)^{k-2}) \vf(e)=\vf(e)^k=\vf(e).
		$$ 
		Therefore, $\varphi: I(X,F)\to (u^{k-1}Bu^{k-1},\bullet_u)$ preserves idempotents. By \cref{Idemp-pres-is-Jordan-homo} it preserves the Jordan product. Since $\ch(F)\ne 2$, $\vf$ is a Jordan homomorphism.
	\end{proof}
	
	
	
	
	
	We specify the result of \cref{l k-potent pres 1} in the case where $1\in \varphi( I(X,F))$.
	\begin{cor} \label{c k-potent pres}
		Let $X$ be a finite poset and $\varphi:I(X,F)\to B$ a linear $k$-potent preserver such that $1\in \varphi( I(X,F))$. Then there exist a Jordan homomorphism $\psi: I(X,F)\to B$ and $u\in C_B(\vf(I(X,F)))$ with $u^{k-1}=1$ such that $\varphi=u\psi$.
	\end{cor}
	\begin{proof}
		As in \cref{l k-potent pres 1} put $u=\vf(\dl)$. Then $u\in C_B(\vf(I(X,F)))$ and it follows from \cref{u^(k-1)vf(e_xy)=vf(e_xy)=vf(e)u^(k-1)} that $u^{k-1}\varphi(f)=\varphi(f)$ for all $f\in I(X,F)$. Since  $1\in \varphi( I(X,F))$  there exists $f$ such that $\varphi(f)=1$. Hence, $u^{k-1}=1$.
		
		Now define $\psi=u^{k-2} \varphi$. In view of \cref{Idemp-pres-is-Jordan-homo} and the fact that $\ch(F)\ne 2$ it suffices to prove that $\psi$ preserves the idempotents of $I(X,F)$. Let $e\in E(I(X,F))$. Then $\vf(e)\in P_k(B)$ and by \cref{vf(1)^n=vf(e)^n+vf(1-e)^n} we have
		\begin{align*}
		\psi(e)^2&=u^{k-2} \varphi(e) u^{k-2} \varphi(e)=u^{k-2} \varphi(e)(\vf(e)^{k-2}+\vf(\dl-e)^{k-2})\varphi(e)\\
		&=u^{k-2} \varphi^k(e)=u^{k-2} \varphi(e)=\psi(e),
		\end{align*}
		as desired. Clearly, $u\psi=u^{k-1}\varphi=\varphi$.  
	\end{proof}

	\begin{cor} \label{c k-potent pres surj}
		Let $X$ be a finite poset and $\varphi:I(X,F)\to I(X,F)$ a surjective linear $k$-potent preserver. Then there exist a Jordan homomorphism $\psi: I(X,F)\to I(X,F)$ and $r\in F$ with $r^{k-1}=1$ such that $\varphi=r  \psi$.
	\end{cor}
	\begin{proof}
		For, $u\in C(I(X,F))$, so $u=r\dl$ in this case.
	\end{proof}
	
	We finally obtain a description of all bijective linear $k$-potent preservers of an incidence algebra. 
	
	\begin{thrm}\label{k-potent-pres-description}
		Let $X$ be a finite connected poset and $F$ a field admitting a $(k-1)$-th primitive root of the unity and such that $\ch(F) \nmid k$. Then any bijective linear $k$-potent preserver of $I(X,F)$ is of the form $r\psi$, where $r\in F$ is such that $r^{k-1}=1$ and $\psi$ is either an automorphism or an anti-automorphism of $I(X,F)$.
	\end{thrm}
	\begin{proof}
		Easily follows from  \cref{c k-potent pres surj,Jord-auto-is-auto-or-anti-auto}.
	\end{proof}

	\begin{rem}
		All the results in this section remain valid when the map $\vf$ preserves $k$-th powers, $|F|\ge k$ and $\ch(F) \nmid k$.
	\end{rem}

	\section*{Appendix: $k$-potents of $I(X,F)$}
	Let $k$ be a fixed integer greater than or equal to $2$. The following results are well-known for $k$-potent matrices (see, for example~\cite{McCloskey84}).
	\begin{lem}\label{from-k-potent-to-idempotent}
		Let $R$ be a ring and $a\in P_k(R)$. If $b=\sum_{i=1}^{k-1}a^i$, then $b^2=(k-1)b$.
	\end{lem}
	\begin{proof}
		Observe that $ab=\sum_{i=1}^{k-1}a^{i+1}=\sum_{i=2}^ka^i=a+\sum_{i=2}^{k-1}a^i=b$. Therefore, $a^ib=b$ for all $i$, so $b^2=\sum_{i=1}^{k-1}a^ib=\sum_{i=1}^{k-1}b=(k-1)b$.
	\end{proof}
	
	\begin{rem}
		If a field $F$ has a primitive $m$-th root of unity, then either $\ch(F)=0$ or $\ch(F)\nmid m$, and thus $m\in F^*$.
		
		Indeed, assume that $\ch(F)=p$ and $p\mid m$. The existence of a primitive $m$-th root of unity implies $m\mid(p^n-1)$, where $p^n=|F|$. Hence, $p\mid (p^n-1)$, a contradiction.
	\end{rem}
	
	\begin{lem}\label{from-k-potent-to-orthogonal-idempotents}
		Let $F$ be a field admitting a primitive $(k-1)$-th root of unity $\ve\in F$. If $A$ is an $F$-algebra and $a\in P_k(A)$, then $b_i:=\frac 1{k-1}\sum_{s=1}^{k-1}\ve^{is}a^s$, $i=1,\dots,k-1$, are pairwise orthogonal idempotents of $A$ such that $a=\sum_{i=1}^{k-1}\ve^{-i} b_i$.
	\end{lem}
	\begin{proof}
		Since $\ve^i a\in P_k(A)$, we obtain $b_i\in E(A)$ from \cref{from-k-potent-to-idempotent}. Now
		\begin{align*}
		b_ia^t=\frac 1{k-1}\left(\sum_{s=1}^t\ve^{i(k-t+s-1)}a^s+\sum_{s=t+1}^{k-1}\ve^{i(s-t)}a^s\right)=\frac {\ve^{-it}}{k-1}\sum_{s=1}^{k-1}\ve^{is}a^s=\ve^{-it}b_i,
		\end{align*}
		so for $1\le i\ne j\le k-1$ we have
		\begin{align*}
		b_ib_j=\left(\sum_{t=1}^{k-1}\ve^{(j-i)t}\right)b_i=\frac{\ve^{(j-i)k}-\ve^{j-i}}{\ve^{j-i}-1}b_i=0.
		\end{align*}
		Finally, 
		\begin{align*}
		\sum_{i=1}^{k-1}\ve^{-i} b_i=\frac 1{k-1}\sum_{i,s=1}^{k-1}\ve^{i(s-1)}a^s=a+\frac 1{k-1}\sum_{s=2}^{k-1}\left(\sum_{i=1}^{k-1}\ve^{i(s-1)}\right)a^s=a,
		\end{align*}
		because $\sum_{i=1}^{k-1}\ve^{i(s-1)}=\frac{\ve^{k(s-1)}-\ve^{s-1}}{\ve^{s-1}-1}=0$ for all $2\le s\le k-1$.
	\end{proof}
	
	\begin{prop}\label{p tripotents conjugate}
		Let $X$ be a locally finite poset and $F$ a field admitting a primitive $(k-1)$-th root of unity in $F$. Then $f\in I(X,F)$ is a $k$-potent if and only if $f$ is conjugate to a diagonal element whose entries are either zero or $(k-1)$-th roots of unity.
	\end{prop}
	\begin{proof}
		The ``if'' part is trivial, so let us prove the ``only if'' part. Take $f\in P_k(I(X,F))$. It suffices to prove that $f$ is conjugate to its diagonal $f_D$. Consider the corresponding orthogonal idempotents $g_i=\frac 1{k-1}\sum_{s=1}^{k-1}\ve^{is}f^s$, $i=1,\dots,k-1$, from \cref{from-k-potent-to-orthogonal-idempotents}. By \cref{diagonalization-idemp} there exists an invertible $\sg\in I(X,F)$ such that $g_i=\sg(f_i)_D\sg\m$ for all $i=1,\dots,k-1$. Thus, 
		$$
		f=\sum_{i=1}^{k-1}\ve^{-i} g_i=\sum_{i=1}^{k-1}\ve^{-i} \sg(g_i)_D\sg\m=\sg\left(\sum_{i=1}^{k-1}\ve^{-i} (g_i)_D\right)\sg\m=\sg f_D\sg\m.
		$$
	\end{proof}
	
	\begin{cor}
		Let $X$ be a locally finite poset and $F$ a field with $\ch(F)\ne 2$. Then $f\in I(X,F)$ is a tripotent if and only if $f$ is conjugate to a diagonal element whose entries are either $0$ or $\pm 1$.
		
		Indeed, $-1$ is the primitive square root of unity in any field $F$ with $\ch(F)\ne 2$.
	\end{cor}

	\begin{exm}
		Let $X=\{x,y\}$ with $x<y$ and $F$ a field with $\ch(F)=2$. Consider $f=\dl+e_{xy}\in I(X,F)$. Then $f^3=f$, but $f_D=\dl$, so $f$ is not conjugate to a diagonal element.
	\end{exm}
	
	
	
	
	\section*{Acknowledgements}
	This work was partially  supported by CNPq 404649/2018-1. The authors are grateful to Professors Manfred Dugas and Daniel Herden for the proof of \cref{diagonalization-idemp}.

	\bibliography{bibl}{}
	\bibliographystyle{acm}

\end{document}